\documentclass[12pt]{amsart}
\usepackage{amsmath,amsthm,amssymb}
\usepackage{mathabx}

\usepackage[utf8]{inputenc} 
\usepackage[T1]{fontenc}    
\usepackage{hyperref}       
\usepackage{url}           
\usepackage{booktabs}       
\usepackage{amsfonts}       
\usepackage{nicefrac}       
\usepackage{microtype}      
\usepackage{xcolor}     
\usepackage{blkarray}

\usepackage[margin=1in]{geometry}
\usepackage{graphicx}
\usepackage{comment}
\usepackage{enumitem}
\setlist[itemize]{leftmargin=*}

\theoremstyle{plain}

\newtheorem{theorem}{Theorem}[section]
\newtheorem{lemma}[theorem]{Lemma}
\newtheorem{corollary}[theorem]{Corollary}
\newtheorem{proposition}[theorem]{Proposition}

\newtheorem{definition}[theorem]{Definition}

\theoremstyle{remark}
\newtheorem{remark}[theorem]{Remark}

\newcommand{\A}{\mathcal{A}}

\numberwithin{equation}{section}

\allowdisplaybreaks

\title[Random Walk with Two Memory Channels]{Phase Transitions for Elephant Random Walks \\ with Two memory Channels} 

\author[K. Maulik]{Krishanu Maulik}
\address{Krishanu Maulik\\
    Theoretical Statistics and Mathematics Unit \\
    Indian Statistical Institute \\
    203 B. T. Road, Kolkata 700108 \\
    West Bengal, India \\
}
\email{kmisical@gmail.com}
\author[P. Roy]{Parthanil Roy}
\address{Parthanil Roy\\
    Department of Mathematics \\
   Indian Institute of Technology Bombay \\
    Powai, Mumbai\\
    Maharashtra 400076 \\
    India 
}
\email{parthanil.roy@gmail.com}
\author[T. Sadhukhan]{Tamojit Sadhukhan}
\address{Tamojit Sadhukhan\\
    Theoretical Statistics and Mathematics Unit  \\
   Indian Statistical Institute \\
    203 B. T. Road, Kolkata 700108 \\
    West Bengal, India
}
\email{tamojit96sadhukhan@gmail.com}
\thanks{The first author is partially supported by a SERB grant MTR/2019/001448. The second author is partially supported by a DST SwarnaJayanti Fellowship. The third author is partially supported by an IMU Breakout Graduate Fellowship.}

\begin{document}

\begin{abstract}
Elephant random walk, introduced to study the effect of memory on random walks, is a novel type of walk that incorporates the information of one randomly chosen past step to determine the future step. However, memory of a process can be multifaceted and can arise due to interactions of more than one underlying phenomena. To model this, random walks with multiple memory channels were introduced in the statistical physics literature by Saha (2022) - here the information on a bunch of independently chosen past steps is needed to decide the future step. With the help of variance heuristics, this work analyzed the two-channel case and predicted a double phase transition: from diffusive to superdiffusive and from superdiffusive to ballistic regimes. We prove these conjectures rigorously (with some corrections), discover a mildly superdiffusive regime at one of the conjectured transition boundaries, and observe a new second-order phase transition. We also carry out a detailed investigation of the asymptotic behavior of the walk at different regimes.
\end{abstract}



\maketitle

\section{Introduction}

Random processes with strong memory arise naturally in various disciplines including physics, economics, biology, geology, etc. Keeping this in mind, \cite{erw1} introduced a new discrete-time stochastic process, where the walker takes $\pm 1$ steps along the set of integers $\mathbb{Z}$. This walk starts at the origin and at the beginning, the walker moves either to the right ($+1$) with probability $q$ or to the left ($-1$) with probability $1 - q$, where $q \in (0,1)$. For all subsequent times, a previous step is selected uniformly at random from the history of the walk. The walker then either replicates this selected step with probability $p$, or takes the opposite step with probability $1 - p$ independently of everything else. The memory parameter $p \in (0,1)$ quantifies the ability of the walker to correctly recall a past step. The walker is metaphorically termed an elephant, which is claimed to possess strong memory. The walk thus generated is called an \emph{elephant random walk}, which has gathered substantial attention in the probability and statistical physics community in recent years; see, e.g., \cite{Bercu_2018, bertoin2022counting, 10.1063/1.4983566, Coletti_2017, coletti2019limit, Coletti_2021, harbola2014memory}.

Random walk with multiple memory channels is a generalization of the elephant random walk, introduced in \cite{rwnmc}, where the walker chooses $m~(\geq 2)$ mutually independent steps (instead of one) uniformly at random from the previous steps, and determines the next step following certain rules based on the chosen steps. As a result, the elephant has multiple memory channels ($m \geq 2$, to be specific), as opposed to the classical elephant random walk with only one such channel, to decide the next step. The strength of all the channels is determined by a single memory parameter $p \in (0,1)$. 

In \cite{rwnmc}, the variance of the position of the walker $S_n$ (at time $n$) was calculated heuristically in the case $m=2$ (i.e. two memory channels) and a double phase transition was predicted as the memory parameter $p$ varies in $(0,1)$, indicating transitions from diffusive to superdiffusive and from superdiffusive to ballistic regimes as described below: 
\begin{align}\label{eq:rw2mc-conj}
    \operatorname{Var}\left(S_n\right) \sim  \begin{cases}
       n & \mbox{if } \quad 0 < p \leq 0.68, \smallskip\\
       n^{8(2p-1)/3} & \mbox{if } \quad 0.68 < p < 0.88,\smallskip\\
       n^2 &\mbox{if }  \quad 0.88 \leq p < 1,
    \end{cases}
\end{align}
as $n \to \infty$. However, an exact probabilistic analysis of the above has remained an open problem, which we have been able to resolve as mentioned below.

We conducted a detailed and rigorous analysis of the asymptotic behavior of the walk using stochastic approximation and martingale techniques. We showed that the scaled position of the random walk converges almost surely to zero for $p < 7/8$ and to a nonzero random variable for $p > 7/8$ (see Theorem~\ref{thm:slln-rw2mc}). This result immediately implies that the walk is ballistic (see Corollary~\ref{cor1}) and transient (see Proposition~\ref{def:trrec})for $p > 7/8$, thus confirming one of the conjectures made in \cite{rwnmc} except that the number $0.88$ had to be corrected to $0.875 = 7/8$. However, the question of almost sure convergence at $p = 7/8$ remains open. We further established that the fluctuations of the scaled position around its almost sure limit exhibit three phase transitions, namely at $p = 11/16$, $p = 7/8$ and $p = ({113+\sqrt{97}})/{128}$ (see Theorem~\ref{thm:clt-rw2mc}). The first two phase transitions were more or less (i.e., with the approximations $0.68 \approx 0.6875 =11/16$ and $0,88 \approx 0.875 = 7/8$) conjectured in \cite{rwnmc}, while the third one is a new finding. For these fluctuations, we rigorously proved central limit theorems (see Theorem~\ref{thm:clt-rw2mc}-a),b)), laws of iterated logarithms (see Theorem~\ref{thm:lil-rw2mc}) and quadratic strong laws (see Theorem~\ref{thm:qsl-rw2mc}) for $p \in (0, 11/16] \cup [({113+\sqrt{97}})/{128}, 1)$ and almost sure convergence results (see Theorem~\ref{thm:clt-rw2mc}-c)) for $p \in (11/16, 7/8) \cup (7/8, ({113+\sqrt{97}})/{128})$.

It is essential to point out the effect of an extra memory channel on the asymptotic properties of the elephant random walk. As is clear from \cite{Bercu_2018, 10.1063/1.4983566, Coletti_2017}, the classical elephant ramdom walk (with one memory channel) undergoes a single phase transition from diffusive to superdiffusive growth at $p=3/4$. Due to the presence of one more memory channel, the random walk introduced in \cite{rwnmc} exhibits a double phase transition, with the same change (from diffusive to superdiffusive regimes) occurs at a smaller value, namely at $p=11/16~(< 3/4)$. Moreover, the additional channel of memory also ensures a nontrivial phase transition at $p=7/8$ from  superdiffusive to ballistic growth, which is absent in the classical elephant random walk except for the trivial case $p=1$. 

The structure of the paper is as follows. Section~\ref{sec:main} details the dynamics of the random walk with two memory channels, our main results (see Section~\ref{sec:res}) and a few open problems (see Remarks~\ref{remark_op1} and~\ref{remark_op2}). We then describe a bunch of some related stochastic processes in Section~\ref{sec:rel_proc}. Finally, Section~\ref{sec:sa} discusses the stochastic approximation techniques, their relevance to the random walk model including the supporting lemmas and their proofs, and the main results are proved in Section~\ref{sec:proofs}.

\section{The Model and Related Processes}\label{sec:main}

In this section, we first define the elephant random walk formally and then introduce the main model, namely the (elephant) random walk with two memory channels. We also state all key results of this paper. 

\subsection{Elephant Random Walk}

{We first describe the elephant random walk as defined in \cite{erw1}. Fix the parameters $p,q \in (0,1)$. We denote the location of the elephant at time $n \geq 0$ by $\xi_n$. Starting from $\xi_0 \equiv 0$, the walk evolves according to $\xi_{n+1} = \xi_{n} + Y_{n+1}$, where for each $n \geq 0$, $Y_{n+1}$ represents the $(n+1)$-th step of the elephant defined as follows: $Y_1 \sim \text{ Rad } (q)$ (i.e., $Y_1$ takes the value $+1$ with probability $q$ and the value $-1$ with probability $1-q$), and for $n \geq 1$, 
\begin{align}\label{eq01}
    Y_{n+1} = \begin{cases}
        +Y_{U_{n+1}} & \text{with probability } p, \\
        -Y_{U_{n+1}} & \text{with probability } 1 - p,
    \end{cases}
\end{align}
with $U_{n+1} \sim \text{Uniform}\{1, 2, \ldots, n\}$ independent of $\{Y_1, Y_2, \ldots, Y_n\}$. Equivalently, \eqref{eq01} can be written as
\begin{align*}
           {Y_{n+1} = \alpha_{n+1}Y_{U_{n+1}}},
\end{align*} 
where {$\alpha_{n+1} {\sim} \text{ Rad }(p)$}, independent of $U_{n+1}$ and $\{Y_1, \ldots, Y_n\}$.} Since the elephant recalls one past step chosen uniformly at random, it effectively operates with a single memory channel governed by the parameter $p$. In particular, through this memory channel, the elephant retrieves a version of the selected past step that is perturbed, through multiplication, by a $\text{Rad }(p)$ random variable $\alpha_{n+1}$.

\subsection{Random Walk with Two Memory Channels}\label{sec:model}

The dynamics of this walk is as follows. For $n \geq 0$, we use $S_n$ and $X_n$ to denote the location of this random walk at time $n$ and the $n$-th step of the walk, respectively. This walk begins from the origin $S_0 \equiv 0$ at epoch $n = 0$. At both epochs $n = 1$ and $n = 2$, the walk moves one step either to the right or to the left, each with probability $1/2$, independently of the other step. The locations of the walker at epochs $n=1$ and $n = 2$ are thus given by 
\[
    S_{1} = S_{0} + X_{1} = X_1, \quad S_{2} = S_{1} + X_{2} = X_1 + X_2,  
\]
where $X_1, X_2 \stackrel{iid}{\sim} \text{ Rad }(1/2)$. For each $n \geq 2$, at epoch $n+1$, the walker chooses two steps with replacement from its past $n$ steps independently and uniformly at random. Let $U_{1,n+1}$ and $U_{2,n+1}$ be the indices of the chosen past steps at epoch $n+1$, so that for $n \geq 2$, 
\[
   U_{1,n+1}, U_{2,n+1}  \stackrel{\text{ iid }}{\sim} \text{ Uniform }\{1, 2, \ldots, n\}, \quad (U_{1,n+1}, U_{2,n+1}) \perp \!\!\! \perp (X_1, X_2, \ldots, X_n).
\]
{Here both the memory channels are governed by a single parameter $p \in (0, 1)$}, and through these two memory channels, the walker is only able to retrieve versions of the two selected past steps, perturbed, through multiplication, by two random variables $\alpha_{1,n+1}$ and $\alpha_{2,n+1}$, respectively, where for $n \geq 2$, \[\alpha_{1,n+1}, \alpha_{2,n+1} \stackrel{\text{ iid }}{\sim} \text{ Rad } (p), \quad (\alpha_{1,n+1}, \alpha_{2,n+1}) \perp \!\!\! \perp (X_1, X_2, \ldots, X_n, U_{1,n+1}, U_{2,n+1}).\] 

The steps retrieved through the two memory channels, which we denote by $X^{(1)}_{n+1}$ and $X^{(2)}_{n+1}$ respectively, are therefore given by \begin{align*}
  X^{(1)}_{n+1} := \alpha_{1,n+1}X_{U_{1,n+1}}, \quad X^{(2)}_{n+1} := \alpha_{2,n+1}X_{U_{2,n+1}}. 
\end{align*}
The location at epoch $n+1$ is $S_{n+1} = S_n + X_{n+1}$, where the $(n+1)$-th step $X_{n+1}$ of the walker is given by 
\begin{align*}
          X_{n+1} = \begin{cases}+1, &\text{ if}\quad  \left(X^{(1)}_{n+1}, X^{(2)}_{n+1}\right) \in \{(+1,+1), (+1, 0), (0,+1)\}, \\[.3cm] \textcolor{white}{+}0, &\text{ if}\quad \left(X^{(1)}_{n+1}, X^{(2)}_{n+1}\right) \in \{(+1,-1), (-1,+1), (0, 0)\}, \\[.3cm] -1, &\text{ if}\quad \left(X^{(1)}_{n+1}, X^{(2)}_{n+1}\right) \in \{ (-1,-1), (-1, 0), (0,-1)\}. \end{cases}
\end{align*}

It is important to note that even though the first two steps are non-zero, the walk is such that the subsequent steps can possibly be zero. In other words, there may be situations where the walker becomes lazy and does not move at all. The $(n+1)$-th step $X_{n+1}$ is described in a different but equivalent way in \cite{rwnmc}, through the random variable $T_{n+1}$ given by
\begin{align*}
   T_{n+1} := \alpha_{1,n+1}X_{U_{1,n+1}} + \alpha_{2,n+1}X_{U_{2,n+1}} \, , \quad n \geq 2.
\end{align*}
It is easy to see that $X_{n+1}$ can also be written as
\begin{align*}
           X_{n+1} = T_{n+1}\boldsymbol{1}_{[T_{n+1} \, \in \, \{-1, 0, +1\}]} + \boldsymbol{1}_{[T_{n+1} =+2]} - \boldsymbol{1}_{[T_{n+1} =-2]}\,.
\end{align*}
In other words, $X_{n+1}$ is nothing but $T_{n+1}$ truncated at $\pm 1$. The above dynamics can be extended naturally to incorporate more than two memory channels as described in \cite{rwnmc}. We shall, however, restrict our attention to the two-channel case.

\subsection{Main Results} \label{sec:res}

In this section, we state the main results for the random walk with two memory channels. Our first result establishes the almost sure convergence of the scaled random walk for all $p \neq 7/8$.

\begin{theorem}[Almost sure convergence]\label{thm:slln-rw2mc}
For all $p \neq 7/8$,
\begin{align} \label{eq:s_p}
\frac{{S}_n}{n} \hspace{.1cm} \stackrel{a.s.}{\to} \hspace{.1cm} \Lambda_p := c_{p}\textnormal{{ Rad }}\left(1/2\right),
\end{align}
where
\begin{align*}
   c_{p} = &{\begin{cases}\hspace{1.2cm}{0}\hspace{0.85cm}, &{p \in \left(0, \frac{7}{8}\right)}, \\[.1cm] {\frac{\sqrt{32p^2-52p+21
   }}{(2p-1)^2}}, &{p \in \left(\frac{7}{8}, 1\right)}.\end{cases}}\end{align*}
\end{theorem}

\noindent An immediate consequence of \eqref{eq:s_p} is the following.

\begin{corollary}\label{cor1}
    If $\frac{7}{8} < p < 1$, 
    ${V}ar\left(S_n\right) \sim n^2$ and so the random walk with two memory channels is ballistic here. 
\end{corollary}

\begin{remark} \label{remark_op1}
    Corollary~\ref{cor1} rigorously verifies one of the predictions of \cite{rwnmc}  (see \eqref{eq:rw2mc-conj}). The case $p=7/8$ is still open. 
\end{remark}

The fluctuations of the walk around the almost sure limit, namely ${S_n}/{n} - \Lambda_{p}$ exhibit different behavior for different values of $p$ and undergoes three phase transitions, respectively, at $ p = {11}/{16}$, $p = {7}/{8}$, and $p = {(113+\sqrt{97})}/{128}$. In order to state these results, we need the notion mentioned below.

\begin{definition}
    Suppose $\Sigma$ is a strictly positive random variable. We say that another random variable $Z$ is a variance mixture of Gaussian with random variance $\Sigma$ (and denote this by ${Z} \sim N\left(0,\Sigma\right)$) if the conditional characteristic function of ${Z}$ given $\Sigma$ is 
    \[
       \mathbb{E}\left(e^{\iota tZ} \mid \Sigma\right) = e^{-\frac{t^2\Sigma}{2}}, \quad t \in \mathbb{R}.
    \]
\end{definition}
    
\noindent With the help of this notion we state the following second-order phase transition result for random walks with two memory channels.

\begin{theorem}[Central Limit Behavior]\label{thm:clt-rw2mc}
The fluctuations of the walk around the almost sure limit exhibit the following central limit behavior based on the value of $p$. 
\begin{enumerate}
\item[a)] For $p \in \left(0, \frac{11}{16}\right)\cup\left(\frac{113+\sqrt{97}}{128}, 1\right)$,
\begin{align}\label{eqthm:clt1}
    \sqrt{n}\left(\frac{{S}_n}{{n}} - \Lambda_p\right) \stackrel{d}{\to} N\left(0,\Sigma^{(1)}_p\right),
\end{align}
where
\begin{align}\label{eq:sigma1}
   \Sigma^{(1)}_p = &{\begin{cases}{\frac{2}{11-16p}}, &{p \in \left(0, \frac{11}{16}\right)}, \\[.1cm] \Sigma^{(1)}_{\alpha, \beta, p}, &{p \in \left(\frac{113+\sqrt{97}}{128}, 1\right)},\end{cases}}
\end{align}
with $\alpha, \beta$ given by \eqref{eq:al-be} and the random variable $\Sigma^{(1)}_{\alpha, \beta, p}$ given by \eqref{eq:sig1}.
\item[b)] For $p \in \left\{\frac{11}{16}, \frac{113+\sqrt{97}}{128}\right\}$,
\begin{align}\label{eqthm:clt2}
    \sqrt{\frac{n}{\log n}}\left(\frac{{S}_n}{{n}} - \Lambda_p\right) \stackrel{d}{\to} N\left(0,\Sigma^{(2)}_p\right),
\end{align}
where
\begin{align}\label{eq:sigma2}
   \Sigma^{(2)}_p = &{\begin{cases}{\frac{2}{3}}, &{p = \frac{11}{16}}, \\[.1cm] \Sigma^{(2)}_{\alpha, \beta, p}, &{p = \frac{113+\sqrt{97}}{128}}.\end{cases}}
\end{align}
with $\alpha, \beta$ given by \eqref{eq:al-be} and the random variable $\Sigma^{(2)}_{\alpha, \beta, p}$ given by \eqref{eq:sig2}.
\item[c)] For $p \in \left(\frac{11}{16}, \frac{7}{8}\right)\cup\left(\frac{7}{8}, \frac{113+\sqrt{97}}{128}\right)$,
\begin{align}\label{eqthm:clt3}
    n^{y_p}\left(\frac{{S}_n}{{n}} - \Lambda_p\right) \stackrel{a.s.}{\to} L_{p},
\end{align}
where
${{L_p}}$ is a finite random variable and 
\begin{align*}
    y_p = \begin{cases}\hspace{1.7cm}{\frac{7-8p}{3}}\hspace{2.05cm},
    &{p \in \left(\frac{11}{16}, \frac{7}{8}\right)}, \smallskip \\{\frac{5p-4-\sqrt{-64p^3+161p^2-134p+37}}{2p-1}}, &{p \in \left(\frac{7}{8}, \frac{113+\sqrt{97}}{128}\right)}.\end{cases}
\end{align*}
\end{enumerate}
\end{theorem}

{For the fluctuations around the almost sure limit, we also derive the laws of iterated logarithms and the quadratic strong laws in the appropriate regimes. It is important to note that in case of the classical elephant random walk, the limits obtained in both of these are always non-random (see, e.g., \cite{gerw} and the references therein), which is not the case for the random walks with two memory channels. This can also be thought of as the effect of an extra memory channel.} 

\begin{theorem}[Law of Iterated Logarithms]\label{thm:lil-rw2mc}
{The rates of almost sure convergence of the fluctuations around the almost sure limit in the appropriate regimes are as follows.
\begin{enumerate}
\item[a)] For $p \in \left(0, \frac{11}{16}\right)$,
\begin{align}\label{eqthm:lil1-a}
    &\limsup_{n \to \infty} \sqrt{\frac{n}{2\log\log n}}\left(\frac{S_n}{{n}} - \Lambda_p\right) \\
    =-&\liminf_{n \to \infty} \sqrt{\frac{n}{2\log\log n}}\left(\frac{S_n}{{n}} - \Lambda_p\right) 
    {= \sqrt{\Sigma^{(1)}_p}}, \nonumber
\end{align}
where $\Sigma^{(1)}_p$ is given by \eqref{eq:sigma1}, and for $p \in \left(\frac{113+\sqrt{97}}{128}, 1\right)$, there exists a constant $C > 0$ such that
\begin{align}\label{eqthm:lil1-b}
    &\limsup_{n \to \infty} \sqrt{\frac{n}{2\log\log n}}\left|\frac{S_n}{{n}} - \Lambda_p\right| \leq C.
\end{align}
\item[b)] For $p \in \left\{\frac{11}{16}, \frac{113+\sqrt{97}}{128}\right\}$,
\begin{align}\label{eqthm:lil2}
    &\limsup_{n \to \infty} \left(\frac{n}{2\log n\log\log\log n}\right)^{1/2}\left(\frac{S_n}{{n}} - \Lambda_p\right) 
    \\ =-&\liminf_{n \to \infty} \left(\frac{n}{2\log n\log\log\log n}\right)^{1/2}\left(\frac{S_n}{{n}} - \Lambda_p\right) 
    {= \sqrt{\Sigma^{(2)}_p}}, \nonumber
\end{align}
where $\Sigma^{(2)}_p$ is given by \eqref{eq:sigma2}.
\end{enumerate}}
\end{theorem} 

{The classical notions of recurrence and transience do not apply directly for the random walk with two memory channels as it exhibits time-inhomogeneous Markovian dynamics. So we formulate it separately for this setup following \cite{gerw}.
}

\begin{definition}[Transience and Recurrence]\label{def:trrec}
    {The random walk with two memory channels is said to be transient (respectively, recurrent) if, starting from the origin, it returns to the origin only finitely many times (respectively, infinitely often) with probability one.
}
\end{definition}

\noindent It is not immediate whether the walk is transient or recurrent or neither. The following result, which is a consequence of Theorem~\ref{thm:slln-rw2mc} and Theorem~\ref{thm:lil-rw2mc}, partially describes the recurrence and transience (as defined in Definition~\ref{def:trrec}) of the random walk with two memory channels.

\begin{proposition}[Transience and Recurrence]
    {The random walk with two memory channels is recurrent for $ 0 < p \leq \frac{11}{16}$ and is transient for $\frac{7}{8} < p < 1$.}
\end{proposition}

\begin{remark} \label{remark_op2}
    The problem of determining whether the random walk with two memory channels is transient or recurrent or neither when $p \in (11/16, 7/8]$ is still open.
\end{remark}

{The final result describes the quadratic strong laws for this model.}

\begin{theorem}[Quadratic Strong Law]\label{thm:qsl-rw2mc}
{We have the following quadratic strong laws for the fluctuations around the almost sure limit.
\begin{enumerate}
\item[a)] For $p \in \left(0, \frac{11}{16}\right)\cup\left(\frac{113+\sqrt{97}}{128}, 1\right)$,
\begin{align}\label{eqthm:qsl1}
    \lim_{n \to \infty} \frac{1}{\log n}\sum_{k=1}^n\left(\frac{S_k}{k}-\Lambda_p\right)^2 
    {= \Sigma^{(1)}_p},
\end{align}
where $\Sigma^{(1)}_p$ is given by \eqref{eq:sigma1}.\\
\item[b)] For $p \in \left\{\frac{11}{16}, \frac{113+\sqrt{97}}{128}\right\}$,
\begin{align}\label{eqthm:qsl2}
    \lim_{n \to \infty}  \frac{1}{\log\log n}\sum_{k=1}^n\left(\frac{1}{\log k}\right)^2\left(\frac{S_k}{k}-\Lambda_p\right)^2 
    {= \Sigma^{(2)}_p},
\end{align}
where $\Sigma^{(2)}_p$ is given by \eqref{eq:sigma2}.
\end{enumerate}}
\end{theorem} 
 
\section{Related Stochastic Processes} \label{sec:rel_proc}

The literature on probability and statistical physics includes several stochastic processes whose dynamics are closely connected to those of random walks with two memory channels. In this section, we describe a bunch of them, one of which will be important for our paper. The other related models can also be investigated using our techniques, even though we refrain from doing so for the sake of brevity of this paper. 

\subsection{A Polya Type Generalized Urn Model with Two Drawings}

{Following \cite{baur2016elephant}, random walk with two memory channels can be embedded into the following Polya type generalized urn model with two drawings; see~\cite{rwnmc}. This urn scheme consists of the discrete-time evolution of balls of three colors, say, red~($R$), black ($B$), and green ($G$). The urn composition at time $n \geq 0$ is denoted by $U_n = (R_n, B_n, G_n)$, where $R_n$, $B_n$, and $G_n$ are the numbers of red, black and green balls, respectively, at epoch $n$. We start with a random urn composition at time $n = 0$, given by,
\[
   U_0 = \begin{cases}
       (2,0,0) &\text{ with probability } 1/4, \\ 
       (0,2,0) &\text{ with probability } 1/4, \\
       (1,1,0) &\text{ with probability } 1/2.
   \end{cases}
\]

At each epoch, two balls are drawn at random, with replacement, from the urn. Then they are returned to the urn along with another new ball of random color. The distribution of the random color depends on the configurations of the drawn balls. After each draw, there are six possible outcomes: $(R,R)$, $(R,B)$, $(R,G)$, $(B,B)$, $(B,G)$, and $(G,G)$. The probabilities of whether the newly added ball is $R$, $B$, or $G$ corresponding to each of the six outcomes are summarized in the following mean replacement matrix:
\[
\begin{blockarray}{cccc}
& R & B & G \\
\begin{block}{c(ccc)}
  (R,R) & p^2 & (1-p)^2 & 2p(1-p) \\
  (R,B) & p(1-p) & p(1-p) & p^2 + (1-p)^2 \\
  (R,G) & p & 1-p & 0 \\
  (B,B) & (1-p)^2 & p^2 & 2p(1-p) \\
  (B,G) & 1-p & p & 0 \\
  (G,G) & 0 & 0 & 1 \\
\end{block}
\end{blockarray}.
 \]
If we associate the $+1$ steps with red balls, the $-1$ steps with black balls, and the zero steps with green balls, then the relation 
\[(S_{n+2} : n \geq 0) \stackrel{d}{=} (R_n - B_n : n \geq 0)\]}
describes the \cite{baur2016elephant}-type embedding mentioned above. We believe that this urn process is of independent interest and can definitely be analyzed using our method.

\subsection{Multidimensional Generalized Elephant Random Walk}\label{eg:mgerw}

The random walk with two memory channels is a special case of the multidimensional generalized elephant random walk proposed by \cite{gerw}. In order to describe this model, let us denote by $[k]$ the set $\{1, 2, \ldots, k\}$ for any positive integer $k$. The {multidimensional generalized elephant random walk} \(\left(\boldsymbol{\xi}_n\right)_{n\geq0}\) is a very general random walk model on \(\mathbb{R}^l\) for some \(l \in \mathbb{N}\). At each time \(n \geq 0\), the position \(\boldsymbol{\xi}_n\) of the walker is given by a linear transformation of the position \(\boldsymbol{\widetilde{\xi}}_n\) of an auxiliary random walk \(\left(\boldsymbol{\widetilde{\xi}}_n\right)_{n\geq0}\) on \([0,\infty)^m\) for \(m \in \mathbb{N}\), according to $\boldsymbol{\xi}_n = \boldsymbol{\A} \boldsymbol{\widetilde{\xi}}_n + n\boldsymbol{\beta}$, where \(\boldsymbol{\A}\) is a nonrandom \(l \times l\) matrix and \(\boldsymbol{\beta}\) is a nonrandom vector in \(\mathbb{R}^l\). The matrix \(\boldsymbol{\A}\) maps the auxiliary position \(\boldsymbol{\widetilde{S}}_n\) into \(\mathbb{R}^l\), while the vector \(n\boldsymbol{\beta}\) introduces a time-dependent drift. 

The auxiliary walk \((\boldsymbol{\widetilde{S}}_n)_{n\geq0}\) is described with several parameters, as follows. For some \(u \in [m+1]\), let \(0 = i_0 < i_1 < i_2 < \cdots < i_{u-2} < i_{u-1} \leq i_u = m\) and define the sets \(v_i = \{i_{j-1}+1, \ldots, i_j\}\) for \(j \in [u]\) with the convention that if \(i_{u-1} = i_u\), then \(v_r = \emptyset\). 
Initiating from the origin \(\boldsymbol{\widetilde{\xi}}_0 = \boldsymbol{0}_m\), the auxiliary walk takes the first step \(\boldsymbol{\widetilde{\xi}}_1 = \boldsymbol{\widetilde{Y}}_1\), a random vector taking values in \(E_m \subseteq [0,\infty)^m\), $E_m$ being an \(m\)-dimensional (possibly unbounded) rectangle with the origin as a corner. Subsequently, for each \(n \geq 1\), $\boldsymbol{\widetilde{\xi}}_{n+1} = \boldsymbol{\widetilde{\xi}}_n + \boldsymbol{\widetilde{Y}}_{n+1}$,
where, conditioned on \(\boldsymbol{\widetilde{Y}}_1, \ldots, \boldsymbol{\widetilde{Y}}_n\), the $(n+1)$-th step \(\boldsymbol{\widetilde{Y}}_{n+1}\) is defined by
\begin{align*}
    \boldsymbol{\widetilde{Y}}_{n+1} = \begin{cases}
        \boldsymbol{Z}_{v_1, n+1} & \text{with probability } {f}_1\left(\boldsymbol{\widetilde{\xi}}_n / n\right), \\
        \boldsymbol{Z}_{v_2, n+1} & \text{with probability } {f}_2\left(\boldsymbol{\widetilde{\xi}}_n / n\right), \\
        \quad\vdots \\
        \boldsymbol{Z}_{v_{r-1}, n+1} & \text{with probability } {f}_{r-1}\left(\boldsymbol{\widetilde{\xi}}_n / n\right), \\
        \boldsymbol{Z}_{v_r, n+1} & \text{with remaining probability } 1 - \sum_{i=1}^{r-1} {f}_i\left(\boldsymbol{\widetilde{\xi}}_n / n\right).
    \end{cases}
\end{align*}
Here \(\boldsymbol{Z}_1, \boldsymbol{Z}_2, \ldots\) are i.i.d. random vectors taking values in \(E_m\), and for each \(n \geq 1\) and \(i \in [u]\), \(\boldsymbol{Z}_{v_i, n+1}\) denotes the vector formed by retaining only the coordinates of \(\boldsymbol{Z}_{n+1}\) in the set \(v_i\), with zeros elsewhere. The functions \(f_i: E_m \rightarrow [0,1]\), for \(i \in [u-1]\), satisfy \(\sum_{i=1}^{u-1} f_i(\boldsymbol{x}) < 1\) for all \(\boldsymbol{x} \in E_m\).

Random walk with two memory channels is a special case of the multidimensional generalized elephant random walk with $l = 1$, $m = 2$, $u = 3$, $i_1 = 1$, $i_2 = i_3 = 2$, $\boldsymbol{\A}_{1\times 2} = \begin{bmatrix} 1 & -1 \end{bmatrix}$, ${\beta}_{1\times 1} = 0$, $\boldsymbol{Z}_1 \equiv \begin{pmatrix}
    1 & 1
\end{pmatrix}^{\top}$, $E_2 = [0,1]^2$ and the functions $f_1, f_2$ given by
\begin{align*}
    f_1\left(x_1, x_2\right) &= \{(1-p)x_1 + px_2 - 1\}^2 - (1-x_1-x_2)^2, \quad (x_1, x_2) \in [0,1]^2,
    \\ f_2\left(x_1, x_2\right) &= \{(1-p)x_2 + px_1 - 1\}^2 - (1-x_1-x_2)^2, \quad (x_1, x_2) \in [0,1]^2.
\end{align*}
Multidimensional generalized elephant random walk will therefore play an important role in this paper.

\subsection{Elephant Random Walk with Multiple Extractions}

Elephant random walk with multiple extractions is a similar model, introduced in \cite{fran}. This is a generalization of the classical elephant random walk in which the $(n+1)$-th step $Z_{n+1}$ is determined as follows. First, the elephant selects \( k \) independent previous steps at random from \( Z_1, \ldots, Z_n \), where \( k \) is an odd integer. If the sum of these selected steps is positive, indicating a majority of positive steps, then \( Z_{n+1} \) is set to \( +1 \) with probability \( p \), and \( -1 \) with probability \( 1 - p \). On the other hand, if the sum is negative, \( Z_{n+1} \) takes the value \( -1 \) with probability \( p \), and \( +1 \) with probability \( 1 - p \). The classical elephant random walk is recovered when \( k = 1 \).

\section{Applications of Stochastic Approximation}\label{sec:sa}

In this section, we derive a relation of the random walk with two memory channels with a class of stochastic processes known as stochastic approximation, using which the asymptotic properties of the former are established. For a overview of stochastic approximation and its connection with random walks, see, e.g. \cite{harold1997stochastic} and \cite{gerw}, respectively. The abovementioned relation is through $\boldsymbol{\Gamma}_n := \begin{pmatrix}{{{S}^+_{n}}/{n}} & {{S}^-_{n}}/{n}\end{pmatrix}^\top$, where ${S}^{\pm}_{n} = \# \pm1 \text{ steps of the walk till time } n$, $n \geq 1$. The following result shows that $\left(\boldsymbol{\Gamma}_n\right)_{n \geq 2}$ is a stochastic approximation process.

\begin{lemma}\label{lem:sa}
    The stochastic process $\boldsymbol{\Gamma}_n$ satisfy the recursion
\begin{align}\label{eq:sa-main}
    \boldsymbol{\Gamma}_{n+1} = \boldsymbol{\Gamma}_{n} + \frac{\boldsymbol{h}_p(\boldsymbol{\Gamma}_{n}) + \boldsymbol{e}_{n+1}}{n+1}, 
    \quad n \geq 2,
\end{align}
where $\boldsymbol{e}_{n+1}$ is a martingale-difference w.r.t.\ the filtration $\mathcal{\boldsymbol{G}}_{n} = \sigma\{\boldsymbol{\Gamma}_1, \ldots, \boldsymbol{\Gamma}_n\}$ and $\boldsymbol{h}_p : [0,1]^2 \to [0,1]^2$ is the following quadratic function
\begin{align*}
    \boldsymbol{h}_p\left(\boldsymbol{x}\right) := \begin{pmatrix}
        \left((1-p)x_1 + px_2 - 1\right)^2 - (1-x_1-x_2)^2 - x_1\\\left((1-p)x_2 + px_1 - 1\right)^2 - (1-x_1-x_2)^2 - x_2
    \end{pmatrix}.
\end{align*}
In other words, $\boldsymbol{\Gamma}_n$ is a stochastic approximation process with drift $\boldsymbol{h}_p$, random noise $\boldsymbol{e}_{n+1}$ and step-size ${1}/{(n+1)}$. For some constant $C > 0$, 
we also have, almost surely, \begin{align}\label{eq:ineq}\mathbb{E}\left(\|\boldsymbol{h}_p(\boldsymbol{\Gamma}_n) + \boldsymbol{e}_{n+1}\|^2 \mid \mathcal{\boldsymbol{G}}_{n}\right) \leq C(1 + \|\boldsymbol{\Gamma}_n\|^2).\end{align}
Additionally, $\boldsymbol{e}_{n+1}$ satisfies the following Lindeberg condition 
    \begin{align}\label{eq:lind}
    \frac{1}{n}\sum_{k=2}^{n}\mathbb{E}\left( \|\boldsymbol{e}_{k+1}\|^2 \mathbf{1}\left\{\|\boldsymbol{e}_{k+1}\| \geq \delta\sqrt{n}\right\}\mid \mathcal{\boldsymbol{G}}_{k}\right) \stackrel{a.s.}{\to} 0, \quad \delta > 0.
    \end{align} 
    Moreover, for all $\epsilon > 0$ such that, almost surely, 
    \begin{align}\label{eq:sup}
        \sup_{n \geq 2} \mathbb{E}\left( \|\boldsymbol{e}_{n+1}\|^{2+\epsilon} \mid \mathcal{\boldsymbol{G}}_{n}\right) < \infty.
\end{align}
\end{lemma}

\begin{proof}
The dynamics of $\boldsymbol{\Gamma}_{n}$ are as follows -- 
\begin{align*}
2\boldsymbol{\Gamma}_2 = \begin{cases}
    \begin{pmatrix}
    2 & 0
\end{pmatrix}^{\top}, &\text{ with probability } \frac{1}{4} \\
\begin{pmatrix}
    1 & 1
\end{pmatrix}^{\top}, &\text{ with probability } \frac{1}{2} \\
\begin{pmatrix}
    0 & 2
\end{pmatrix}^{\top}, &\text{ with probability } \frac{1}{4},
\end{cases}
\end{align*}
and for $n \geq 2$, 
\begin{align}\label{eq-epochn}
(n+1)\boldsymbol{\Gamma}_{n+1} = n\boldsymbol{\Gamma}_{n} + \begin{cases}
    \begin{pmatrix}
    1 & 0
\end{pmatrix}^{\top}, &\text{ with probability } \begin{pmatrix}
    1 & 0
\end{pmatrix}^{\top}\left(\boldsymbol{h}_p\left(\boldsymbol{\Gamma}_{n}\right) + \boldsymbol{\Gamma}_{n}\right) \\
\begin{pmatrix}
    0 & 1
\end{pmatrix}^{\top}, &\text{ with probability } \begin{pmatrix}
    0 & 1
\end{pmatrix}^{\top}\left(\boldsymbol{h}_p\left(\boldsymbol{\Gamma}_{n}\right) + \boldsymbol{\Gamma}_{n}\right) \\
\begin{pmatrix}
    0 & 0
\end{pmatrix}^{\top}, &\text{ with probability } 1 - \begin{pmatrix}
    1 & 1
\end{pmatrix}^{\top}\left(\boldsymbol{h}_p\left(\boldsymbol{\Gamma}_{n}\right) + \boldsymbol{\Gamma}_{n}\right).
\end{cases}
\end{align}
For $n \geq 2$, define $\boldsymbol{e}_{n+1} := (n+1)\left(\boldsymbol{\Gamma}_{n+1} - \boldsymbol{\Gamma}_{n}\right) - \boldsymbol{h}_p\left(\boldsymbol{\Gamma}_{n}\right)$ so that $\boldsymbol{e}_{n+1}$ is a martingale-difference w.r.t.\ the filtration $\mathcal{\boldsymbol{G}}_{n} = \sigma\{\boldsymbol{\Gamma}_1, \ldots, \boldsymbol{\Gamma}_n\}$ and \eqref{eq:sa-main} holds. From \eqref{eq-epochn}, the definition of $\boldsymbol{h}_p$ and $\boldsymbol{e}_{n+1}$, \eqref{eq:ineq}, \eqref{eq:lind} and \eqref{eq:sup} follow.
\end{proof}

We denote the Jacobian of $\boldsymbol{h}_p$ at $\boldsymbol{x} \in [0,1]^2$ by $\boldsymbol{J}_p\left(\boldsymbol{x}\right)$. For $i = 1,2$, $\lambda^{(i)}_p\left(\boldsymbol{x}\right)$ and $\boldsymbol{\nu}^{(i)}_p\left(\boldsymbol{x}\right)$ denote the $i^{th}$ largest eigenvalue of $\boldsymbol{J}_p\left(\boldsymbol{x}\right)$ and the corresponding eigenvector, respectively. 
The irrational conjugate of ${x} \in [0,1]$ is denoted by $\widebar{{x}}$ and $\widebar{\boldsymbol{x}} := \begin{pmatrix}
    \bar{x}_1 & \bar{x}_2
\end{pmatrix}$. We next establish the almost sure convergence of $\boldsymbol{\Gamma}_n$ for $p \neq 7/8$. The same for $p = 7/8$ is still open. 

\begin{theorem}\label{thm:sa-slln}
    For $p \neq 7/8$, $\boldsymbol{\Gamma}_n\stackrel{a.s.}{\to} \boldsymbol{\Gamma}^{(p)}$. For $p < 7/8$, $\boldsymbol{\Gamma}^{(p)}$ is degenerate at $\boldsymbol{\gamma}_0 := \begin{pmatrix} {1}/{3} & {1}/{3}\end{pmatrix}^{\top}$. For $p > 7/8$, $\boldsymbol{\Gamma}^{(p)}$ takes the values $\boldsymbol{\gamma}_p$ and $\widebar{\boldsymbol{\gamma}}_p$, where
\begin{align*}
    \boldsymbol{\gamma}_p = \begin{pmatrix}
       \frac{8p^2-10p+3+\sqrt{32p^2-52p+21
   }}{2(2p-1)^2} & \frac{8p^2-10p+3-\sqrt{32p^2-52p+21
   }}{2(2p-1)^2} 
    \end{pmatrix}.
\end{align*}
Further, for $\boldsymbol{\gamma}$ in the support of $\boldsymbol{\Gamma}^{(p)}$, on the event $\left\{\boldsymbol{\Gamma}^{(p)} = \boldsymbol{\gamma}\right\}$, we have 
\begin{align}\label{eq:sigma-conv}
    \mathbb{E}\left(\boldsymbol{e}_{n+1}\boldsymbol{e}^{\top}_{n+1} \mid \mathcal{\boldsymbol{G}}_{n}\right) \stackrel{a.s.}{\to} \boldsymbol{\sigma}\left(\boldsymbol{\gamma}\right) := \begin{pmatrix}
      \gamma_1 - \gamma_1^2  & -\gamma_1\gamma_2\\  -\gamma_1\gamma_2 & \gamma_2 - \gamma_2^2
    \end{pmatrix}.
\end{align}
\end{theorem}

\begin{proof}
The asymptotic behavior of $\boldsymbol{\Gamma}_n$ is determined by the ordinary differential equation 
\begin{align}\label{eq:ode}
   \frac{d\boldsymbol{x}(t)}{dt} = \boldsymbol{h}_p(\boldsymbol{x}(t)), \quad t \geq 0.
\end{align}
It is easy to verify that $\boldsymbol{h}_p$ satisfies the {Dulac's Criterion} and so {Poincare-Bendixon Theorem} implies that trajectories of \eqref{eq:ode} converge to the zeros of $\boldsymbol{h}_p$ as $t \to \infty$.
If $p<7/8$, $\boldsymbol{h}_p$ has two zeros in $[0,1]^2$ -- $\boldsymbol{0}$ and $\boldsymbol{\gamma}_0$. 
If $p > 7/8$, $\boldsymbol{h}_p$ and has four zeros in $[0,1]^2$ -- $\boldsymbol{0}$, $\boldsymbol{\gamma}_0$, $\boldsymbol{\gamma}_p$ and $\widebar{\boldsymbol{\gamma}}_p$. Also, an immidiate calculation shows that
\begin{align}\label{eq:evalues}
     &\lambda^{(1)}_p\left(\boldsymbol{0}\right) = 1, \quad \lambda^{(2)}_p\left(\boldsymbol{0}\right) = 4p-3, \\
     &\lambda^{(1)}_p\left(\boldsymbol{\gamma}_0\right) = \begin{cases} -1,  & \hspace{.8cm} p \leq \frac{1}{2}, \\ -\frac{1}{3}\left(7-8p\right), &\frac{1}{2} < p < \frac{7}{8},\end{cases} \quad \lambda^{(2)}_p\left(\boldsymbol{\gamma}_0\right) = \begin{cases} -\frac{1}{3}\left(7-8p\right),  & \hspace{.8cm} p \leq \frac{1}{2}, \\ -1, &\frac{1}{2} < p < \frac{7}{8},\end{cases} \nonumber \\
      &\lambda^{(1)}_p\left(\boldsymbol{\gamma}_p\right) = \widebar{\lambda}^{(2)}_p\left(\boldsymbol{\gamma}_p\right) = \lambda^{(1)}_p\left(\boldsymbol{\gamma}_p\right) = \widebar{\lambda}^{(2)}_p\left(\boldsymbol{\gamma}_p\right) = \frac{4-5p+\sqrt{-64p^3+161p^2-134p+37}}{2p-1} \nonumber.
\end{align}
From \eqref{eq:evalues}, it follows that, for $p < 7/8$, $\boldsymbol{0}$ is linearly unstable and $\boldsymbol{\gamma}_0$ is linearly stable for \eqref{eq:ode}, whereas for $p > 7/8$, $\boldsymbol{0}$, $\boldsymbol{\gamma}_0$ are linearly unstable and $\boldsymbol{\gamma}_p$, $\widebar{\boldsymbol{\gamma}}_p$ are linearly stable for \eqref{eq:ode} (see, for example, Assumption A2 of \cite{kaniovski1995strong} for the definitions of linearly unstable and linearly stable zeros). Theorem 1 of \cite{pemantle1990nonconvergence} then implies 
\[
   \mathbb{P}\left(\boldsymbol{\Gamma}_n \to \boldsymbol{\gamma_0}\right) = 0, \quad p > \frac{7}{8}, \quad \mathbb{P}\left(\boldsymbol{\Gamma}_n \to \boldsymbol{0}\right) = 0, \quad p \neq \frac{7}{8}.
\]
Consequently, from Theorem 1 of \cite{kaniovski1995strong}, we get that 
\[ \boldsymbol{\Gamma}_n \stackrel{a.s.}{\to} \boldsymbol{\Gamma}^{(p)} \text{ with } \operatorname{supp}\left(\boldsymbol{\Gamma}^{(p)}\right) = \begin{cases}\left\{\boldsymbol{\gamma}_0\right\}, \quad &p < \frac{7}{8}\\\left\{\boldsymbol{\gamma}_p, \widebar{\boldsymbol{\gamma}}_p\right\}, \quad &p > \frac{7}{8}.\end{cases} \]
This concludes the proof.
\end{proof}

Note that for $p < 7/8$,
\begin{align}\label{eq:eienvec1}
    \boldsymbol{\nu}^{(1)}_p\left(\boldsymbol{\gamma}_0\right) = \begin{cases} \begin{pmatrix}
            +1 \\ +1
        \end{pmatrix},  & \hspace{.8cm} p \leq \frac{1}{2}, \\ \begin{pmatrix}
            +1 \\ -1
        \end{pmatrix}, &\frac{1}{2} < p < \frac{7}{8},\end{cases} \quad
    \boldsymbol{\nu}^{(2)}_p\left(\boldsymbol{\gamma}_0\right) = \begin{cases} \begin{pmatrix}
            +1 \\ -1
        \end{pmatrix},  & \hspace{.8cm} p \leq \frac{1}{2}, \\ \begin{pmatrix}
            +1 \\ +1
        \end{pmatrix}, &\frac{1}{2} < p < \frac{7}{8},\end{cases}
\end{align}
and for $p > 7/8$,
\begin{align}\label{eq:eienvec2}
        \boldsymbol{\nu}^{(1)}_p\left(\boldsymbol{\gamma}_p\right) &= \boldsymbol{\nu}^{(1)}_p\left(\widebar{\boldsymbol{\gamma}}_p\right) = \begin{pmatrix}
        +1 \\ -\frac{(1-p)\sqrt{32p^2-52p+21
   }-\sqrt{-64p^3+161p^2-134p+37}}{4-5p+p\sqrt{32p^2-52p+21
   }}   
        \end{pmatrix},
        \\
        {\boldsymbol{\nu}}^{(2)}_p\left(\boldsymbol{\gamma}_p\right) &= \boldsymbol{\nu}^{(2)}_p\left(
            \widebar{\boldsymbol{\gamma}}_p\right) = \begin{pmatrix}
         +1\\-\frac{(1-p)\sqrt{32p^2-52p+21
   }+\sqrt{-64p^3+161p^2-134p+37}}{4-5p+p\sqrt{32p^2-52p+21
   }}   
        \end{pmatrix}.
\end{align}
From \eqref{eq:eienvec1} and \eqref{eq:eienvec2} (respectively, \eqref{eq:evalues}), we get  that $\boldsymbol{\nu}^{(i)}_p\left(\boldsymbol{\Gamma}^{(p)}\right)$ (respectively, $\boldsymbol{\lambda}^{(i)}_p\left(\boldsymbol{\Gamma}^{(p)}\right)$) is actually non-random, for $i = 1,2$ and $p \neq 7/8$. The next result shows the behavior of $\boldsymbol{\Gamma}_n - \boldsymbol{\Gamma}^{(p)}$ when $\lambda^{(1)}_p\left(\boldsymbol{\Gamma}^{(p)}\right) < - \frac{1}{2}$ almost surely.

\begin{theorem}
Let $c_1, c_2 \in \mathbb{R}$ and $p \neq 7/8$ be such that $\lambda^{(1)}_p\left(\boldsymbol{\Gamma}^{(p)}\right) < - \frac{1}{2}$ almost surely. 
 \begin{itemize}
    \item[i)] The following central limit theorem holds - 
    \begin{align}\label{eq:mclt1}  &\sqrt{n}\left\{\left(c_1\boldsymbol{\nu}^{(1)}_p\left(\boldsymbol{\Gamma}^{(p)}\right) + c_2\boldsymbol{\nu}^{(2)}_p\left(\boldsymbol{\Gamma}^{(p)}\right)\right)^\top\left(\boldsymbol{\Gamma}_n - \boldsymbol{\Gamma}^{(p)}\right)\right\} \stackrel{d}{\to} N\left(0, {\Sigma}^{(1)}_{c_1,c_2,p}\right),
    \end{align}
    where
    {\begin{align}\label{eq:sig1}
        {\Sigma}^{(1)}_{c_1,c_2,p} := \left(c_1\boldsymbol{\nu}^{(1)}_p\left(\boldsymbol{\Gamma}^{(p)}\right) + c_2\boldsymbol{\nu}^{(2)}_p\left(\boldsymbol{\Gamma}^{(p)}\right)\right)^{\top}\left\{\int_{0}^{\infty}e^{t\left(\boldsymbol{J}_p\left(\boldsymbol{\Gamma}^{(p)}\right)+\frac{1}{2}\boldsymbol{I}\right)}\sigma\left(\boldsymbol{\Gamma}^{(p)}\right)e^{t\left(\boldsymbol{J}_p\left(\boldsymbol{\Gamma}^{(p)}\right)+\frac{1}{2}\boldsymbol{I}\right)^{\top}} dt\right\}\nonumber\\\times\left(c_1\boldsymbol{\nu}^{(1)}_p\left(\boldsymbol{\Gamma}^{(p)}\right) + c_2\boldsymbol{\nu}^{(2)}_p\left(\boldsymbol{\Gamma}^{(p)}\right)\right).
    \end{align}}
    \item[ii)] For some absolute constant $C > 0$, almost surely  
    \begin{align}\label{eq:mlil11}  &\limsup_{n \to \infty}\sqrt{\frac{n}{2\log\log n}}\left\|\left(c_1\boldsymbol{\nu}^{(1)}_p\left(\boldsymbol{\Gamma}^{(p)}\right) + c_2\boldsymbol{\nu}^{(2)}_p\left(\boldsymbol{\Gamma}^{(p)}\right)\right)^\top\left(\boldsymbol{\Gamma}_n - \boldsymbol{\Gamma}^{(p)}\right)\right\| \leq C,
    \end{align}
    and, almost surely, \begin{align}\label{eq:mlil12}  &\limsup_{n \to \infty}\sqrt{\frac{n}{2\log\log n}}\left\{\left(c_1\boldsymbol{\nu}^{(1)}_p\left(\boldsymbol{\Gamma}^{(p)}\right)\right)^\top\left(\boldsymbol{\Gamma}_n - \boldsymbol{\Gamma}^{(p)}\right)\right\} \\ = - &\liminf_{n \to \infty}\sqrt{\frac{n}{2\log\log n}}\left\{\left(c_1\boldsymbol{\nu}^{(1)}_p\left(\boldsymbol{\Gamma}^{(p)}\right)\right)^\top\left(\boldsymbol{\Gamma}_n - \boldsymbol{\Gamma}^{(p)}\right)\right\} = \sqrt{{\Sigma}^{(1)}_{c_1,0,p}}, \nonumber
    \end{align}
    and, also, almost surely, \begin{align}\label{eq:mlil13}  &\limsup_{n \to \infty}\sqrt{\frac{n}{2\log\log n}}\left\{\left(c_2\boldsymbol{\nu}^{(2)}_p\left(\boldsymbol{\Gamma}^{(p)}\right)\right)^\top\left(\boldsymbol{\Gamma}_n - \boldsymbol{\Gamma}^{(p)}\right)\right\} \\ = - &\liminf_{n \to \infty}\sqrt{\frac{n}{2\log\log n}}\left\{\left(c_2\boldsymbol{\nu}^{(2)}_p\left(\boldsymbol{\Gamma}^{(p)}\right)\right)^\top\left(\boldsymbol{\Gamma}_n - \boldsymbol{\Gamma}^{(p)}\right)\right\} = \sqrt{{\Sigma}^{(1)}_{0,c_2,p}}. \nonumber
    \end{align}
    \item[iii)] The following quadratic strong law holds almost surely -  
    \begin{align}\label{eq:mqsl1}  &\lim_{n \to \infty} \frac{1}{\log n}\sum_{k=2}^n\left\{\left(c_1\boldsymbol{\nu}^{(1)}_p\left(\boldsymbol{\Gamma}^{(p)}\right) + c_2\boldsymbol{\nu}^{(2)}_p\left(\boldsymbol{\Gamma}^{(p)}\right)\right)^\top\left(\boldsymbol{\Gamma}_k - \boldsymbol{\Gamma}^{(p)}\right)\right. \\ &\hspace{3cm}\times\left.\left(\boldsymbol{\Gamma}_k - \boldsymbol{\Gamma}^{(p)}\right)^{\top}\left(c_1\boldsymbol{\nu}^{(1)}_p\left(\boldsymbol{\Gamma}^{(p)}\right) + c_2\boldsymbol{\nu}^{(2)}_p\left(\boldsymbol{\Gamma}^{(p)}\right)\right)\right\} = {\Sigma}^{(1)}_{c_1,c_2,p}. \nonumber
    \end{align}
    \end{itemize}    
\end{theorem}

\begin{proof}
To show \eqref{eq:mclt1} and \eqref{eq:mlil11}, \eqref{eq:mlil12}, \eqref{eq:mlil13}, \eqref{eq:mqsl1}, respectively, we show, for any $\boldsymbol{\gamma} \in \operatorname{supp}\left(\boldsymbol{\Gamma}^{(p)}\right)$, that on the event $\left\{\boldsymbol{\Gamma}^{(p)} = \boldsymbol{\gamma}\right\}$, 
\begin{align}\label{eq:mclt1-a}  
&\sqrt{n}\left(\boldsymbol{\Gamma}_n - \boldsymbol{\gamma}\right) \stackrel{d}{\to} N\left(0, \int_{0}^{\infty}e^{t\left(\boldsymbol{J}_p\left(\boldsymbol{\gamma}\right)+\frac{1}{2}\boldsymbol{I}\right)}\sigma\left(\boldsymbol{\gamma}\right)e^{t\left(\boldsymbol{J}_p\left(\boldsymbol{\gamma}\right)+\frac{1}{2}\boldsymbol{I}\right)^{\top}} dt\right),
\end{align} 
and, almost surely,
\begin{align}\label{eq:mlil11-a}  
&\limsup_{n \to \infty}\sqrt{\frac{n}{2\log\log n}}\left\|\boldsymbol{\Gamma}_n - \boldsymbol{\gamma}\right\| \leq C,
\end{align}
\begin{align}\label{eq:mlil12-a}  
&\limsup_{n \to \infty}\sqrt{\frac{n}{2\log\log n}}\left(\boldsymbol{\nu}^{(1)}_p\left(\boldsymbol{\gamma}\right)\right)^\top\left(\boldsymbol{\Gamma}_n - \boldsymbol{\gamma}\right)\\ = - &\liminf_{n \to \infty}\sqrt{\frac{n}{2\log\log n}}\left(\boldsymbol{\nu}^{(1)}_p\left(\boldsymbol{\gamma}\right)\right)^\top\left(\boldsymbol{\Gamma}_n - \boldsymbol{\Gamma}^{(p)}\right) \nonumber \\ = \textcolor{white}{-}&\left(\left(\boldsymbol{\nu}^{(1)}_p\left(\boldsymbol{\gamma}\right) \right)^{\top}\left\{\int_{0}^{\infty}e^{t\left(\boldsymbol{J}_p\left(\boldsymbol{\gamma}\right)+\frac{1}{2}\boldsymbol{I}\right)}\sigma\left(\boldsymbol{\gamma}\right)e^{t\left(\boldsymbol{J}_p\left(\boldsymbol{\gamma}\right)+\frac{1}{2}\boldsymbol{I}\right)^{\top}} dt\right\}\left(\boldsymbol{\nu}^{(1)}_p\left(\boldsymbol{\gamma}\right) \right)\right)^{1/2}, \nonumber
\end{align}
\begin{align}\label{eq:mlil13-a}  
&\limsup_{n \to \infty}\sqrt{\frac{n}{2\log\log n}}\left(\boldsymbol{\nu}^{(2)}_p\left(\boldsymbol{\gamma}\right)\right)^\top\left(\boldsymbol{\Gamma}_n - \boldsymbol{\gamma}\right) \\ = - &\liminf_{n \to \infty}\sqrt{\frac{n}{2\log\log n}}\left(\boldsymbol{\nu}^{(2)}_p\left(\boldsymbol{\gamma}\right)\right)^\top\left(\boldsymbol{\gamma}_n - \boldsymbol{\gamma}\right)\nonumber \\ = \textcolor{white}{-}&\left(\left(\boldsymbol{\nu}^{(2)}_p\left(\boldsymbol{\gamma}\right) \right)^{\top}\left\{\int_{0}^{\infty}e^{t\left(\boldsymbol{J}_p\left(\boldsymbol{\gamma}\right)+\frac{1}{2}\boldsymbol{I}\right)}\sigma\left(\boldsymbol{\gamma}\right)e^{t\left(\boldsymbol{J}_p\left(\boldsymbol{\gamma}\right)+\frac{1}{2}\boldsymbol{I}\right)^{\top}} dt\right\}\left(\boldsymbol{\nu}^{(2)}_p\left(\boldsymbol{\gamma}\right) \right)\right)^{1/2}, \nonumber
\end{align}
\begin{align}\label{eq:mqsl1-a}  
&\lim_{n \to \infty} \frac{1}{\log n}\sum_{k=2}^n\left(\boldsymbol{\Gamma}_k - \boldsymbol{\gamma}\right)\left(\boldsymbol{\Gamma}_k - \boldsymbol{\gamma}\right)^{\top} = \int_{0}^{\infty}e^{t\left(\boldsymbol{J}_p\left(\boldsymbol{\gamma}\right)+\frac{1}{2}\boldsymbol{I}\right)}\sigma\left(\boldsymbol{\gamma}\right)e^{t\left(\boldsymbol{J}_p\left(\boldsymbol{\gamma}\right)+\frac{1}{2}\boldsymbol{I}\right)^{\top}} dt. 
\end{align}
From Lemma~\ref{lem:sa} and Lemma~\ref{thm:sa-slln}, it follows that 
\begin{itemize}
    \item[(a)] The stochastic approximation $\boldsymbol{\Gamma}_n$ satisfies Equation (2.1) of \cite{10.1214/aoap/1027961032} and Equation (1) of \cite{PELLETIER1998217}, with the residual term being identically zero.
    \item[(b)] The drift $\boldsymbol{h}_p$ satisfies Assumption A1.1 of \cite{10.1214/aoap/1027961032} (with $\rho \equiv 1$) and Assumption (A1) of \cite{PELLETIER1998217}.
    \item[(c)] The random noise $\boldsymbol{e}_{n+1}$ satisfies Assumption A1.2 of \cite{10.1214/aoap/1027961032} and Assumption (A3) of \cite{PELLETIER1998217}.
    \item[(d)] The step-size $1/(n+1)$) satisfies Assumption A1.3 of \cite{10.1214/aoap/1027961032} and Assumption (A2) of \cite{PELLETIER1998217}, with $v(t) = t$ and $\gamma_0 = 1$.
\end{itemize}
Consequently, Theorem 1 of \cite{10.1214/aoap/1027961032} (also see Remark (a) just after Theorem 1 in \cite{10.1214/aoap/1027961032} and Equation (7) in \cite{PELLETIER1998217}) implies \eqref{eq:mclt1-a}, proving \eqref{eq:mclt1}. Similarly, Theorem 1 (respectively, Theorem 3) of \cite{PELLETIER1998217} implies \eqref{eq:mlil11-a}, \eqref{eq:mlil12-a} and \eqref{eq:mlil13-a} (respectively, \eqref{eq:mqsl1-a}) proving \eqref{eq:mlil11}, \eqref{eq:mlil12} and \eqref{eq:mlil13} (respectively, \eqref{eq:mqsl1}).
\end{proof}

To establish the rate of convergence in Theorem~\ref{thm:sa-slln} when $\lambda^{(1)}_p\left(\boldsymbol{\Gamma}^{(p)}\right) \geq - \frac{1}{2}$ almost surely, we need some preliminary results. For $n \geq 2$ and $\boldsymbol{x} \in [0,1]^2$, define
\begin{align}\label{eq:not1}
    a_n &:= \sum_{k=2}^n\frac{1}{k+1}, \quad \boldsymbol{M}_{n+1}\left(\boldsymbol{x}\right) := \sum_{k=2}^n\frac{e^{(a_n-a_k)\boldsymbol{J}_p\left(\boldsymbol{x}\right)}}{k+1}\boldsymbol{e}_{k+1}, \quad \boldsymbol{\Delta}_{n+1}\left(\boldsymbol{x}\right) := \boldsymbol{\Gamma}_{n+1} - \boldsymbol{x} - \boldsymbol{M}_{n+1}\left(\boldsymbol{x}\right).\end{align} 
We need the asymptotic properties of $\boldsymbol{\nu}^{(i)}_p\left(\boldsymbol{\gamma}\right)^{\top}\boldsymbol{M}_{n+1}\left(\boldsymbol{\Gamma}^{(p)}\right)$ and $\boldsymbol{\nu}^{(i)}_p\left(\boldsymbol{\gamma}\right)^{\top}\boldsymbol{\Delta}_{n+1}\left(\boldsymbol{\Gamma}^{(p)}\right)$ for $i = 1,2$ and  different values of $p \neq 7/8$, depending on the value of $\boldsymbol{\lambda}^{(i)}_p\left(\boldsymbol{\Gamma}^{(p)}\right)$. For any $\boldsymbol{\gamma} \in \operatorname{supp}\left(\boldsymbol{\Gamma}^{(p)}\right)$, $i = 1,2$ and $p \neq 7/8$, as $\boldsymbol{\nu}^{(i)}_p\left(\boldsymbol{\gamma}\right)$ is an eigenvector of $\boldsymbol{J}_p\left(\boldsymbol{\gamma}\right)$ corresponding to the eigenvalue $\boldsymbol{\lambda}^{(i)}_p\left(\boldsymbol{\gamma}\right)$, it follows that
\begin{align}\label{eq:eigen-mg}
    \boldsymbol{\nu}^{(i)}_p\left(\boldsymbol{\gamma}\right)^\top\boldsymbol{M}_{k+1}\left(\boldsymbol{\gamma}\right) &= \sum_{k=2}^n\frac{e^{(a_n-a_k)\boldsymbol{\lambda}^{(i)}_p\left(\boldsymbol{\gamma}\right)}}{k+1}\boldsymbol{\nu}^{(i)}_p\left(\boldsymbol{\gamma}\right)^{\top}\boldsymbol{e}_{k+1} \nonumber
    \\ &= e^{a_n\boldsymbol{\lambda}^{(i)}_p\left(\boldsymbol{\gamma}\right)}\sum_{k=2}^n\Phi^{(i,p)}_{k}\left(\boldsymbol{\gamma}\right)\boldsymbol{\nu}^{(i)}_p\left(\boldsymbol{\gamma}\right)^{\top}\boldsymbol{e}_{k+1}, \text{ where } \Phi^{(i,p)}_{k}\left(\boldsymbol{\gamma}\right) = \frac{e^{-a_k\boldsymbol{\lambda}^{(i)}_p\left(\boldsymbol{\gamma}\right)}}{k+1}. 
\end{align}
For $\boldsymbol{\gamma} \in \operatorname{supp}\left(\boldsymbol{\Gamma}^{(p)}\right)$, $i = 1,2$ and $p \neq 7/8$, let \begin{align}\label{eq:c}c_{i,p}\left(\boldsymbol{\gamma}\right) = \lim_{n\to\infty} e^{(a_n-\log n)\boldsymbol{\lambda}^{(i)}_p\left(\boldsymbol{\gamma}\right)}.\end{align}
\begin{lemma}\label{lem:1}
Fix $i \in \{1,2\}$. Let $p \neq 7/8$ be such that $\lambda^{(i)}_p\left(\boldsymbol{\Gamma}^{(p)}\right) < - \frac{1}{2}$ almost surely.  
    \begin{itemize}
    \item[i)] The following central limit theorem holds - 
    \begin{align}\label{eq:clt1}  &\sqrt{n}\left\{{{\boldsymbol{\nu}^{(i)}_p\left(\boldsymbol{\Gamma}^{(p)}\right)^\top}\boldsymbol{M}_{n+1}\left(\boldsymbol{\Gamma}^{(p)}\right)}\right\} \stackrel{d}{\to} N\left(0, \frac{\boldsymbol{\nu}^{(i)}_p\left(\boldsymbol{\Gamma}^{(p)}\right)^\top\boldsymbol{\sigma\left(\boldsymbol{\Gamma}^{(p)}\right)}\boldsymbol{\nu}^{(i)}_p\left(\boldsymbol{\Gamma}^{(p)}\right)}{-\left(2\lambda^{(i)}_p\left(\boldsymbol{\Gamma}^{(p)}\right)+1\right)}\right). 
    \end{align}
    \item[ii)] The following law of iterated logarithms holds almost surely -  
    \begin{align}\label{eq:lil1}  &\limsup_{n \to \infty}\sqrt{\frac{n}{2\log\log n}}\left\{{{\boldsymbol{\nu}^{(i)}_p\left(\boldsymbol{\Gamma}^{(p)}\right)^\top}\boldsymbol{M}_{n+1}\left(\boldsymbol{\Gamma}^{(p)}\right)}\right\} \\ = - &\liminf_{n \to \infty}\sqrt{\frac{n}{2\log\log n}}\left\{{{\boldsymbol{\nu}^{(i)}_p\left(\boldsymbol{\Gamma}^{(p)}\right)^\top}\boldsymbol{M}_{n+1}\left(\boldsymbol{\Gamma}^{(p)}\right)}\right\} \nonumber \\ =\textcolor{white}{+}&  \left(\frac{\boldsymbol{\nu}^{(i)}_p\left(\boldsymbol{\Gamma}^{(p)}\right)^\top\boldsymbol{\sigma\left(\boldsymbol{\Gamma}^{(p)}\right)}\boldsymbol{\nu}^{(i)}_p\left(\boldsymbol{\Gamma}^{(p)}\right)}{-\left(2\lambda^{(i)}_p\left(\boldsymbol{\Gamma}^{(p)}\right)+1\right)}\right)^{1/2}. \nonumber
    \end{align}
    \item[iii)] The following quadratic strong law holds almost surely -  
    \begin{align}\label{eq:qsl1}  &\lim_{n \to \infty} \frac{1}{\log n}\sum_{k=2}^n\left\{{{\boldsymbol{\nu}^{(i)}_p\left(\boldsymbol{\Gamma}^{(p)}\right)^\top}\boldsymbol{M}_{k+1}\left(\boldsymbol{\Gamma}^{(p)}\right)}\boldsymbol{M}_{k+1}\left(\boldsymbol{\Gamma}^{(p)}\right)^{\top}\boldsymbol{\nu}^{(i)}_p\left(\boldsymbol{\Gamma}^{(p)}\right)\right\} \\ &= \frac{\boldsymbol{\nu}^{(i)}_p\left(\boldsymbol{\Gamma}^{(p)}\right)^\top\boldsymbol{\sigma\left(\boldsymbol{\Gamma}^{(p)}\right)}\boldsymbol{\nu}^{(i)}_p\left(\boldsymbol{\Gamma}^{(p)}\right)}{-\left(2\lambda^{(i)}_p\left(\boldsymbol{\Gamma}^{(p)}\right)+1\right)}. \nonumber
    \end{align}
    \end{itemize}
\end{lemma}

\begin{proof}  
To show \eqref{eq:clt1}, we show, for any $\boldsymbol{\gamma} \in \operatorname{supp}\left(\boldsymbol{\Gamma}^{(p)}\right)$, that on the event $\left\{\boldsymbol{\Gamma}^{(p)} = \boldsymbol{\gamma}\right\}$, 
\begin{align}\label{eq:clt1-a}  
&\sqrt{n}\left\{{{\boldsymbol{\nu}^{(i)}_p\left(\boldsymbol{\gamma}\right)^\top}\boldsymbol{M}_{n+1}\left(\boldsymbol{\gamma}\right)}\right\} \stackrel{d}{\to} N\left(0, \frac{\boldsymbol{\nu}^{(i)}_p\left(\boldsymbol{\gamma}\right)^\top\boldsymbol{\sigma\left(\boldsymbol{\gamma}\right)}\boldsymbol{\nu}^{(i)}_p\left(\boldsymbol{\gamma}\right)}{{-\left(2\lambda^{(i)}_p\left(\boldsymbol{\gamma}\right)+1\right)}}\right). 
\end{align}
From \eqref{eq:eigen-mg}, \eqref{eq:sigma-conv} and \eqref{eq:lind}, it follows that $\left(e^{-a_n\boldsymbol{\lambda}^{(i)}_p\left(\boldsymbol{\gamma}\right)}{\boldsymbol{\nu}^{(i)}_p\left(\boldsymbol{\gamma}\right)^\top}\boldsymbol{M}_{n+1}\left(\boldsymbol{\gamma}\right)\right)_{n \geq 2}$ is a square-integrable martingale with quadratic variations $\left\langle{e^{-a_n\boldsymbol{\lambda}^{(i)}_p\left(\boldsymbol{\gamma}\right)}{\boldsymbol{\nu}^{(i)}_p\left(\boldsymbol{\gamma}\right)^\top}\boldsymbol{M}\left(\boldsymbol{\gamma}\right)}\right\rangle_{n}$ saisfying
\begin{align*}
\frac{{n}^{2\lambda^{(i)}_p\left(\boldsymbol{\gamma}\right)+1}}{\left(c_{i,p}\left(\boldsymbol{\gamma}\right)\right)^2 }\left\langle{e^{-a_n\boldsymbol{\lambda}^{(i)}_p\left(\boldsymbol{\gamma}\right)}{\boldsymbol{\nu}^{(i)}_p\left(\boldsymbol{\gamma}\right)^\top}\boldsymbol{M}\left(\boldsymbol{\gamma}\right)}\right\rangle_{n} \stackrel{a.s.}{\to} \frac{\boldsymbol{\nu}^{(i)}_p\left(\boldsymbol{\gamma}\right)^\top\boldsymbol{\sigma\left(\boldsymbol{\gamma}\right)}\boldsymbol{\nu}^{(i)}_p\left(\boldsymbol{\gamma}\right)}{{-\left(2\lambda^{(i)}_p\left(\boldsymbol{\gamma}\right)+1\right)}},
\end{align*}
and for all $\delta > 0$, $\Delta\left(e^{-a_n\boldsymbol{\lambda}^{(i)}_p\left(\boldsymbol{\gamma}\right)}{\boldsymbol{\nu}^{(i)}_p\left(\boldsymbol{\gamma}\right)^\top}\boldsymbol{M}_{n+1}\left(\boldsymbol{\gamma}\right)\right) := e^{-a_n\boldsymbol{\lambda}^{(i)}_p\left(\boldsymbol{\gamma}\right)}{\boldsymbol{\nu}^{(i)}_p\left(\boldsymbol{\gamma}\right)^\top}\boldsymbol{M}_{n+1}\left(\boldsymbol{\gamma}\right) - $\\$e^{-a_{n-1}\boldsymbol{\lambda}^{(i)}_p\left(\boldsymbol{\gamma}\right)}{\boldsymbol{\nu}^{(i)}_p\left(\boldsymbol{\gamma}\right)^\top}\boldsymbol{M}_{n}\left(\boldsymbol{\gamma}\right)$ satisfies
\begin{align*}
    &\frac{{n}^{2\lambda^{(i)}_p\left(\boldsymbol{\gamma}\right)+1}}{\left(c_{i,p}\left(\boldsymbol{\gamma}\right)\right)^2 }\sum_{k=2}^{n}\mathbb{E}\left( \left\|\Delta\left({\boldsymbol{\nu}^{(i)}_p\left(\boldsymbol{\gamma}\right)^\top}\boldsymbol{M}_{n+1}\left(\boldsymbol{\gamma}\right)\right)\right\|^2\right. \\ &\hspace{3cm}\times\left.\mathbf{1}\left\{\left\|\Delta\left({\boldsymbol{\nu}^{(i)}_p\left(\boldsymbol{\gamma}\right)^\top}\boldsymbol{M}_{n+1}\left(\boldsymbol{\gamma}\right)\right)\right\| \geq \delta{n}^{-\lambda^{(i)}_p\left(\boldsymbol{\gamma}\right)-1/2}\right\}\mid \mathcal{\boldsymbol{G}}_{k}\right) \stackrel{a.s.}{\to} 0.
\end{align*}
Hence from Corollary 2.1.10 of \cite{duflo2013random}, we get 
\begin{align}\label{eq:clt1-b}
    \frac{{n}^{\lambda^{(i)}_p\left(\boldsymbol{\gamma}\right)+1/2}}{c_{i,p}\left(\boldsymbol{\gamma}\right)}e^{-a_n\boldsymbol{\lambda}^{(i)}_p\left(\boldsymbol{\gamma}\right)}{\boldsymbol{\nu}^{(i)}_p\left(\boldsymbol{\gamma}\right)^\top}\boldsymbol{M}_{n+1}\left(\boldsymbol{\gamma}\right) \stackrel{d}{\to}  N\left(0, \frac{\boldsymbol{\nu}^{(i)}_p\left(\boldsymbol{\gamma}\right)^\top\boldsymbol{\sigma\left(\boldsymbol{\gamma}\right)}\boldsymbol{\nu}^{(i)}_p\left(\boldsymbol{\gamma}\right)}{{-\left(2\lambda^{(i)}_p\left(\boldsymbol{\gamma}\right)+1\right)}}\right).
\end{align}
From \eqref{eq:c} and \eqref{eq:clt1-b}, \eqref{eq:clt1-a} follows, establishing \eqref{eq:clt1}.

To show \eqref{eq:lil1}, we show, for any $\boldsymbol{\gamma} \in \operatorname{supp}\left(\boldsymbol{\Gamma}^{(p)}\right)$, that on the event $\left\{\boldsymbol{\Gamma}^{(p)} = \boldsymbol{\gamma}\right\}$, almost surely,
\begin{align}\label{eq:lil1-a}  
\limsup_{n \to \infty}\sqrt{\frac{n}{2\log\log n}}\left\{{{\boldsymbol{\nu}^{(i)}_p\left(\boldsymbol{\gamma}\right)^\top}\boldsymbol{M}_{n+1}\left(\boldsymbol{\gamma}\right)}\right\}  &= - \liminf_{n \to \infty}\sqrt{\frac{n}{2\log\log n}}\left\{{{\boldsymbol{\nu}^{(i)}_p\left(\boldsymbol{\gamma}\right)^\top}\boldsymbol{M}_{n+1}\left(\boldsymbol{\gamma}\right)}\right\} \nonumber \\ &=\sqrt{\frac{\boldsymbol{\nu}^{(i)}_p\left(\boldsymbol{\gamma}\right)^\top\boldsymbol{\sigma\left(\boldsymbol{\gamma}\right)}\boldsymbol{\nu}^{(i)}_p\left(\boldsymbol{\gamma}\right)}{-\left(2\lambda^{(i)}_p\left(\boldsymbol{\gamma}\right)+1\right)}}. 
\end{align}
From \eqref{eq:eigen-mg}, Lemma~\ref{lem:sa} (in particular, \eqref{eq:sigma-conv} and \eqref{eq:sup}), it follows that $\boldsymbol{\nu}^{(i)}_p\left(\boldsymbol{\gamma}\right)^{\top}\boldsymbol{e}_{n+1}$ is a martingale-difference w.r.t.\ the filtration $\mathcal{\boldsymbol{G}}_{n}$ such that, there exists $0 < \delta < 1$ for which almost surely,
\begin{align*}
        &\hspace{3cm}\sup_{n \geq 2} \mathbb{E}\left( \|\boldsymbol{e}_{n+1}\|^{2(1+\delta)} \mid \mathcal{\boldsymbol{G}}_{n}\right) < \infty, \\
&\mathbb{E}\left(\boldsymbol{\nu}^{(i)}_p\left(\boldsymbol{\gamma}\right)^{\top}\boldsymbol{e}_{n+1}\boldsymbol{e}^{\top}_{n+1}\boldsymbol{\nu}^{(i)}_p\left(\boldsymbol{\gamma}\right) \mid \mathcal{\boldsymbol{G}}_{n}\right) \stackrel{a.s.}{\to} \boldsymbol{\nu}^{(i)}_p\left(\boldsymbol{\gamma}\right)^{\top}\boldsymbol{\sigma}\left(\boldsymbol{\gamma}\right)\boldsymbol{\nu}^{(i)}_p\left(\boldsymbol{\gamma}\right),
\end{align*}
and $\tau^{(i,p)}_n\left(\boldsymbol{\gamma}\right) := \sum_{k=2}^n\left(\Phi^{(i,p)}_{k}\left(\boldsymbol{\gamma}\right)\right)^2$ satisfies
\begin{align}\label{eq:lil1-c}
    &\hspace{2.75cm}{n^{2\lambda^{(i)}_p\left(\boldsymbol{\gamma}\right)+1}}{\tau^{(i,p)}_n\left(\boldsymbol{\gamma}\right)} \to -\frac{\left((c_{i,p}\left(\boldsymbol{\gamma}\right)\right)^2}{2\lambda^{(i)}_p\left(\boldsymbol{\gamma}\right)+1}, \\
    &\sum_{k=2}^n\left(\frac{\left(\Phi^{(i,p)}_{k}\left(\boldsymbol{\gamma}\right)\right)^2}{\tau^{(i,p)}_k\left(\boldsymbol{\gamma}\right)}\right)^{1+\delta} < \infty, \quad
    \frac{\Phi^{(i,p)}_{k}\left(\boldsymbol{\gamma}\right)\left(\log \log {\tau^{(i,p)}_k\left(\boldsymbol{\gamma}\right)}\right)^{1/\delta}}{\tau^{(i,p)}_k\left(\boldsymbol{\gamma}\right)} \to 0. \nonumber
\end{align}
Hence Stout’s law of the iterated logarithm (see, for example, Result 1 of \cite{PELLETIER1998217}) implies that, almost surely,
\begin{align}\label{eq:lil1-b}  
&\limsup_{n \to \infty}\sqrt{\frac{1}{2\tau^{(i,p)}_n\log\log \tau^{(i,p)}_n}}\left\{\sum_{k=2}^n\Phi^{(i,p)}_{k}\left(\boldsymbol{\gamma}\right)\boldsymbol{\nu}^{(i)}_p\left(\boldsymbol{\gamma}\right)^{\top}\boldsymbol{e}_{k+1}\right\} \nonumber \\ = - &\liminf_{n \to \infty}\sqrt{\frac{1}{2\tau^{(i,p)}_n\log\log \tau^{(i,p)}_n}}\left\{\sum_{k=2}^n\Phi^{(i,p)}_{k}\left(\boldsymbol{\gamma}\right)\boldsymbol{\nu}^{(i)}_p\left(\boldsymbol{\gamma}\right)^{\top}\boldsymbol{e}_{k+1}\right\} \nonumber \\ =\textcolor{white}{+}&\sqrt{{\boldsymbol{\nu}^{(i)}_p\left(\boldsymbol{\gamma}\right)^\top\boldsymbol{\sigma\left(\boldsymbol{\gamma}\right)}\boldsymbol{\nu}^{(i)}_p\left(\boldsymbol{\gamma}\right)}}. 
\end{align}
From \eqref{eq:eigen-mg}, \eqref{eq:c}, \eqref{eq:lil1-c} and \eqref{eq:lil1-b}, \eqref{eq:lil1-a} follows, establishing \eqref{eq:lil1}.

To show \eqref{eq:qsl1}, we show, for any $\boldsymbol{\gamma} \in \operatorname{supp}\left(\boldsymbol{\Gamma}^{(p)}\right)$, that on the event $\left\{\boldsymbol{\Gamma}^{(p)} = \boldsymbol{\gamma}\right\}$, almost surely,
\begin{align}\label{eq:qsl1-a}  
&\lim_{n \to \infty} \frac{1}{\log n}\sum_{k=2}^n\left\{{{\boldsymbol{\nu}^{(i)}_p\left(\boldsymbol{\gamma}\right)^\top}\boldsymbol{M}_{k+1}\left(\boldsymbol{\gamma}\right)}\boldsymbol{M}_{k+1}\left(\boldsymbol{\gamma}\right)^{\top}\boldsymbol{\nu}^{(i)}_p\left(\boldsymbol{\gamma}\right)\right\} = \frac{\boldsymbol{\nu}^{(i)}_p\left(\boldsymbol{\gamma}\right)^\top\boldsymbol{\sigma\left(\boldsymbol{\gamma}\right)}\boldsymbol{\nu}^{(i)}_p\left(\boldsymbol{\gamma}\right)}{-\left(2\lambda^{(i)}_p\left(\boldsymbol{\gamma}\right)+1\right)}. 
\end{align}
Note that
\[
   \frac{n\left(\Phi^{(i,p)}_{n}\left(\boldsymbol{\gamma}\right)\right)^2}{\tau^{(i,p)}_n\left(\boldsymbol{\gamma}\right)} \to -\left(2\lambda^{(i)}_p\left(\boldsymbol{\gamma}\right)+1\right).
\]
So, Theorem 3 of \cite{BERCU2004157} implies that 
\begin{align}\label{eq:qsl1-b}  
&\lim_{n \to \infty} \frac{1}{\log \tau^{(i,p)}_n\left(\boldsymbol{\gamma}\right)}\sum_{k=2}^n\frac{\left(\Phi^{(i,p)}_{k}\left(\boldsymbol{\gamma}\right)\right)^2}{\tau^{(i,p)}_k\left(\boldsymbol{\gamma}\right)}\frac{\left(\sum_{j=2}^k\Phi^{(i,p)}_{j}\left(\boldsymbol{\gamma}\right)\boldsymbol{\nu}^{(i)}_p\left(\boldsymbol{\gamma}\right)^{\top}\boldsymbol{e}_{j+1}\right)^2}{\tau^{(i,p)}_{k-1}\left(\boldsymbol{\gamma}\right)} = {\boldsymbol{\nu}^{(i)}_p\left(\boldsymbol{\gamma}\right)^\top\boldsymbol{\sigma\left(\boldsymbol{\gamma}\right)}\boldsymbol{\nu}^{(i)}_p\left(\boldsymbol{\gamma}\right)}. 
\end{align}
From \eqref{eq:eigen-mg}, \eqref{eq:c}, \eqref{eq:lil1-c} and \eqref{eq:qsl1-b}, \eqref{eq:qsl1-a} follows, establishing \eqref{eq:qsl1}.
\end{proof}

\begin{lemma}\label{lem:2}
    Fix $i \in \{1,2\}$. Let $p \neq 7/8$ be such that $\lambda^{(i)}_p\left(\boldsymbol{\Gamma}^{(p)}\right) = - \frac{1}{2}$ almost surely.  
    \begin{itemize}
    \item[i)] The following central limit theorem holds - 
      \begin{align}\label{eq:clt2}  &\sqrt{\frac{n}{\log n}}\left\{{{\boldsymbol{\nu}^{(i)}_p\left(\boldsymbol{\Gamma}^{(p)}\right)^\top}\boldsymbol{M}_{n+1}\left(\boldsymbol{\Gamma}^{(p)}\right)}\right\} \stackrel{d}{\to} N\left(0, {\boldsymbol{\nu}^{(i)}_p\left(\boldsymbol{\Gamma}^{(p)}\right)^\top\boldsymbol{\sigma\left(\boldsymbol{\Gamma}^{(p)}\right)}\boldsymbol{\nu}^{(i)}_p\left(\boldsymbol{\Gamma}^{(p)}\right)}\right). 
     \end{align}
    \item[ii)] The following law of iterated logarithms holds almost surely -
    \begin{align}\label{eq:lil2}  &\limsup_{n \to \infty}\sqrt{\frac{n}{2\log n \log\log\log n}}\left\{{{\boldsymbol{\nu}^{(i)}_p\left(\boldsymbol{\Gamma}^{(p)}\right)^\top}\boldsymbol{M}_{n+1}\left(\boldsymbol{\Gamma}^{(p)}\right)}\right\} \\ = - &\liminf_{n \to \infty}\sqrt{\frac{n}{2\log n\log\log\log n}}\left\{{{\boldsymbol{\nu}^{(i)}_p\left(\boldsymbol{\Gamma}^{(p)}\right)^\top}\boldsymbol{M}_{n+1}\left(\boldsymbol{\Gamma}^{(p)}\right)}\right\} \nonumber \\ =\textcolor{white}{+}& \sqrt{\boldsymbol{\nu}^{(i)}_p\left(\boldsymbol{\Gamma}^{(p)}\right)^\top\boldsymbol{\sigma\left(\boldsymbol{\Gamma}^{(p)}\right)}\boldsymbol{\nu}^{(i)}_p\left(\boldsymbol{\Gamma}^{(p)}\right)}. \nonumber
    \end{align}
    \item[iii)] The following quadratic strong law holds almost surely -  
    \begin{align}\label{eq:qsl2}  &\lim_{n \to \infty} \frac{1}{\log \log n}\sum_{k=2}^n\left(\frac{1}{\log k}\right)^2\left\{{{\boldsymbol{\nu}^{(i)}_p\left(\boldsymbol{\Gamma}^{(p)}\right)^\top}\boldsymbol{M}_{k+1}\left(\boldsymbol{\Gamma}^{(p)}\right)}\boldsymbol{M}_{k+1}\left(\boldsymbol{\Gamma}^{(p)}\right)^{\top}\boldsymbol{\nu}^{(i)}_p\left(\boldsymbol{\Gamma}^{(p)}\right)\right\} \\ &= {\boldsymbol{\nu}^{(i)}_p\left(\boldsymbol{\Gamma}^{(p)}\right)^\top\boldsymbol{\sigma\left(\boldsymbol{\Gamma}^{(p)}\right)}\boldsymbol{\nu}^{(i)}_p\left(\boldsymbol{\Gamma}^{(p)}\right)}. \nonumber
    \end{align}
    \end{itemize}
\end{lemma}

\begin{proof}
    To show \eqref{eq:clt2}, we show, for any $\boldsymbol{\gamma} \in \operatorname{supp}\left(\boldsymbol{\Gamma}^{(p)}\right)$, that on the event $\left\{\boldsymbol{\Gamma}^{(p)} = \boldsymbol{\gamma}\right\}$, 
\begin{align}\label{eq:clt2-a}  
&\sqrt{\frac{n}{\log n}}\left\{{{\boldsymbol{\nu}^{(i)}_p\left(\boldsymbol{\gamma}\right)^\top}\boldsymbol{M}_{n+1}\left(\boldsymbol{\gamma}\right)}\right\} \stackrel{d}{\to} N\left(0, {\boldsymbol{\nu}^{(i)}_p\left(\boldsymbol{\gamma}\right)^\top\boldsymbol{\sigma\left(\boldsymbol{\gamma}\right)}\boldsymbol{\nu}^{(i)}_p\left(\boldsymbol{\gamma}\right)}\right). 
\end{align}
From \eqref{eq:eigen-mg}, \eqref{eq:sigma-conv} and \eqref{eq:lind}, it follows that $\left(e^{-a_n\boldsymbol{\lambda}^{(i)}_p\left(\boldsymbol{\gamma}\right)}{\boldsymbol{\nu}^{(i)}_p\left(\boldsymbol{\gamma}\right)^\top}\boldsymbol{M}_{n+1}\left(\boldsymbol{\gamma}\right)\right)_{n \geq 2}$ is a square-integrable martingale with quadratic variations $\left\langle{e^{-a_n\boldsymbol{\lambda}^{(i)}_p\left(\boldsymbol{\gamma}\right)}{\boldsymbol{\nu}^{(i)}_p\left(\boldsymbol{\gamma}\right)^\top}\boldsymbol{M}\left(\boldsymbol{\gamma}\right)}\right\rangle_{n}$ satisfying
\begin{align*}
\frac{1}{\left(c_{i,p}\left(\boldsymbol{\gamma}\right)\right)^2\log n }\left\langle{e^{-a_n\boldsymbol{\lambda}^{(i)}_p\left(\boldsymbol{\gamma}\right)}{\boldsymbol{\nu}^{(i)}_p\left(\boldsymbol{\gamma}\right)^\top}\boldsymbol{M}\left(\boldsymbol{\gamma}\right)}\right\rangle_{n} \stackrel{a.s.}{\to} {\boldsymbol{\nu}^{(i)}_p\left(\boldsymbol{\gamma}\right)^\top\boldsymbol{\sigma\left(\boldsymbol{\gamma}\right)}\boldsymbol{\nu}^{(i)}_p\left(\boldsymbol{\gamma}\right)},
\end{align*}
and for all $\delta > 0$, $\Delta\left(e^{-a_n\boldsymbol{\lambda}^{(i)}_p\left(\boldsymbol{\gamma}\right)}{\boldsymbol{\nu}^{(i)}_p\left(\boldsymbol{\gamma}\right)^\top}\boldsymbol{M}_{n+1}\left(\boldsymbol{\gamma}\right)\right) := e^{-a_n\boldsymbol{\lambda}^{(i)}_p\left(\boldsymbol{\gamma}\right)}{\boldsymbol{\nu}^{(i)}_p\left(\boldsymbol{\gamma}\right)^\top}\boldsymbol{M}_{n+1}\left(\boldsymbol{\gamma}\right) - $ \\ $e^{-a_n\boldsymbol{\lambda}^{(i)}_p\left(\boldsymbol{\gamma}\right)}{\boldsymbol{\nu}^{(i)}_p\left(\boldsymbol{\gamma}\right)^\top}\boldsymbol{M}_{n}\left(\boldsymbol{\gamma}\right)$ satisfies
\begin{align*}
    &\frac{1}{\left(c_{i,p}\left(\boldsymbol{\gamma}\right)\right)^2 \log n}\sum_{k=2}^{n}\mathbb{E}\left( \left\|\Delta\left({\boldsymbol{\nu}^{(i)}_p\left(\boldsymbol{\gamma}\right)^\top}\boldsymbol{M}_{n+1}\left(\boldsymbol{\gamma}\right)\right)\right\|^2\right.& \\ &\hspace{4cm}\left.\mathbf{1}\left\{\left\|\Delta\left({\boldsymbol{\nu}^{(i)}_p\left(\boldsymbol{\gamma}\right)^\top}\boldsymbol{M}_{n+1}\left(\boldsymbol{\gamma}\right)\right)\right\| \geq \delta\sqrt{\log n}\right\}\mid \mathcal{\boldsymbol{G}}_{k}\right) \stackrel{a.s.}{\to} 0.
\end{align*}
Hence from Corollary 2.1.10 of \cite{duflo2013random}, we get 
\begin{align}\label{eq:clt2-b}
    \frac{1}{c_{i,p}\left(\boldsymbol{\gamma}\right)\log n}e^{-a_n\boldsymbol{\lambda}^{(i)}_p\left(\boldsymbol{\gamma}\right)}{\boldsymbol{\nu}^{(i)}_p\left(\boldsymbol{\gamma}\right)^\top}\boldsymbol{M}_{n+1}\left(\boldsymbol{\gamma}\right) \stackrel{d}{\to}  N\left(0, {\boldsymbol{\nu}^{(i)}_p\left(\boldsymbol{\gamma}\right)^\top\boldsymbol{\sigma\left(\boldsymbol{\gamma}\right)}\boldsymbol{\nu}^{(i)}_p\left(\boldsymbol{\gamma}\right)}\right).
\end{align}
From \eqref{eq:c} and \eqref{eq:clt2-b}, \eqref{eq:clt2-a} follows, establishing \eqref{eq:clt2}.

To show \eqref{eq:lil2}, we show, for any $\boldsymbol{\gamma} \in \operatorname{supp}\left(\boldsymbol{\Gamma}^{(p)}\right)$, that on the event $\left\{\boldsymbol{\Gamma}^{(p)} = \boldsymbol{\gamma}\right\}$, almost surely,
\begin{align}\label{eq:lil2-a}  
&\limsup_{n \to \infty}\sqrt{\frac{n}{2\log n \log\log\log n}}\left\{{{\boldsymbol{\nu}^{(i)}_p\left(\boldsymbol{\gamma}\right)^\top}\boldsymbol{M}_{n+1}\left(\boldsymbol{\gamma}\right)}\right\} \\ = - &\liminf_{n \to \infty}\sqrt{\frac{n}{2\log n\log\log\log n}}\left\{{{\boldsymbol{\nu}^{(i)}_p\left(\boldsymbol{\gamma}\right)^\top}\boldsymbol{M}_{n+1}\left(\boldsymbol{\gamma}\right)}\right\} \nonumber \\ =\textcolor{white}{+}& \sqrt{\boldsymbol{\nu}^{(i)}_p\left(\boldsymbol{\gamma}\right)^\top\boldsymbol{\sigma\left(\boldsymbol{\gamma}\right)}\boldsymbol{\nu}^{(i)}_p\left(\boldsymbol{\gamma}\right)}. \nonumber
\end{align}
From \eqref{eq:eigen-mg}, Lemma~\ref{lem:sa} (in particular, \eqref{eq:sigma-conv} and \eqref{eq:sup}), it follows that $\boldsymbol{\nu}^{(i)}_p\left(\boldsymbol{\gamma}\right)^{\top}\boldsymbol{e}_{n+1}$ is a martingale-difference w.r.t.\ the filtration $\mathcal{\boldsymbol{G}}_{n}$ such that, there exists $0 < \delta < 1$ for which almost surely
\begin{align*}
        &\hspace{3cm}\sup_{n \geq 2} \mathbb{E}\left( \|\boldsymbol{e}_{n+1}\|^{2(1+\delta)} \mid \mathcal{\boldsymbol{G}}_{n}\right) < \infty, \\
&\mathbb{E}\left(\boldsymbol{\nu}^{(i)}_p\left(\boldsymbol{\gamma}\right)^{\top}\boldsymbol{e}_{n+1}\boldsymbol{e}^{\top}_{n+1}\boldsymbol{\nu}^{(i)}_p\left(\boldsymbol{\gamma}\right) \mid \mathcal{\boldsymbol{G}}_{n}\right) \stackrel{a.s.}{\to} \boldsymbol{\nu}^{(i)}_p\left(\boldsymbol{\gamma}\right)^{\top}\boldsymbol{\sigma}\left(\boldsymbol{\gamma}\right)\boldsymbol{\nu}^{(i)}_p\left(\boldsymbol{\gamma}\right),
\end{align*}
and $\tau^{(i,p)}_n\left(\boldsymbol{\gamma}\right) := \sum_{k=2}^n\left(\Phi^{(i,p)}_{k}\left(\boldsymbol{\gamma}\right)\right)^2$ satisfies
\begin{align}\label{eq:lil2-c}
    &\hspace{2.75cm}\frac{\tau^{(i,p)}_n\left(\boldsymbol{\gamma}\right)}{\log n} \to -{\left(c_{i,p}\left(\boldsymbol{\gamma}\right)\right)^2}, \\
    &\sum_{k=2}^n\left(\frac{\left(\Phi^{(i,p)}_{k}\left(\boldsymbol{\gamma}\right)\right)^2}{\tau^{(i,p)}_k\left(\boldsymbol{\gamma}\right)}\right)^{1+\delta} < \infty, \quad
    \frac{\Phi^{(i,p)}_{k}\left(\boldsymbol{\gamma}\right)\left(\log \log {\tau^{(i,p)}_k\left(\boldsymbol{\gamma}\right)}\right)^{1/\delta}}{\tau^{(i,p)}_k\left(\boldsymbol{\gamma}\right)} \to 0. \nonumber
\end{align}
Hence Stout’s law of the iterated logarithm (see, for example, Result 1 of \cite{PELLETIER1998217}) implies that, almost surely,
\begin{align}\label{eq:lil2-b}  
&\limsup_{n \to \infty}\sqrt{\frac{1}{2\tau^{(i,p)}_n\log\log \tau^{(i,p)}_n}}\left\{\sum_{k=2}^n\Phi^{(i,p)}_{k}\left(\boldsymbol{\gamma}\right)\boldsymbol{\nu}^{(i)}_p\left(\boldsymbol{\gamma}\right)^{\top}\boldsymbol{e}_{k+1}\right\} \nonumber \\ = - &\liminf_{n \to \infty}\sqrt{\frac{1}{2\tau^{(i,p)}_n\log\log \tau^{(i,p)}_n}}\left\{\sum_{k=2}^n\Phi^{(i,p)}_{k}\left(\boldsymbol{\gamma}\right)\boldsymbol{\nu}^{(i)}_p\left(\boldsymbol{\gamma}\right)^{\top}\boldsymbol{e}_{k+1}\right\} \nonumber \\ =\textcolor{white}{+}&\sqrt{{\boldsymbol{\nu}^{(i)}_p\left(\boldsymbol{\gamma}\right)^\top\boldsymbol{\sigma\left(\boldsymbol{\gamma}\right)}\boldsymbol{\nu}^{(i)}_p\left(\boldsymbol{\gamma}\right)}}. 
\end{align}
From \eqref{eq:c}, \eqref{eq:lil2-c} and \eqref{eq:lil2-b}, \eqref{eq:lil2-a} follows, establishing \eqref{eq:lil2}.

To show \eqref{eq:qsl2}, we show, for any $\boldsymbol{\gamma} \in \operatorname{supp}\left(\boldsymbol{\Gamma}^{(p)}\right)$, that on the event $\left\{\boldsymbol{\Gamma}^{(p)} = \boldsymbol{\gamma}\right\}$, almost surely,
\begin{align}\label{eq:qsl2-a}  
&\lim_{n \to \infty} \frac{1}{\log \log n}\sum_{k=2}^n\left(\frac{1}{\log k}\right)^2\left\{{{\boldsymbol{\nu}^{(i)}_p\left(\boldsymbol{\gamma}\right)^\top}\boldsymbol{M}_{k+1}\left(\boldsymbol{\gamma}\right)}\boldsymbol{M}_{k+1}\left(\boldsymbol{\gamma}\right)^{\top}\boldsymbol{\nu}^{(i)}_p\left(\boldsymbol{\gamma}\right)\right\} \\ &= {\boldsymbol{\nu}^{(i)}_p\left(\boldsymbol{\gamma}\right)^\top\boldsymbol{\sigma\left(\boldsymbol{\gamma}\right)}\boldsymbol{\nu}^{(i)}_p\left(\boldsymbol{\gamma}\right)}. \nonumber
\end{align}
Note that
\[
   \frac{n\log n\left(\Phi^{(i,p)}_{n}\left(\boldsymbol{\gamma}\right)\right)^2}{\tau^{(i,p)}_n\left(\boldsymbol{\gamma}\right)} \to 1.
\]
So, Theorem 3 of \cite{BERCU2004157} implies that 
\begin{align}\label{eq:qsl2-b}  
&\lim_{n \to \infty} \frac{1}{\log \tau^{(i,p)}_n\left(\boldsymbol{\gamma}\right)}\sum_{k=2}^n\frac{\left(\Phi^{(i,p)}_{k}\left(\boldsymbol{\gamma}\right)\right)^2}{\tau^{(i,p)}_k\left(\boldsymbol{\gamma}\right)}\frac{\left(\sum_{j=2}^k\Phi^{(i,p)}_{j}\left(\boldsymbol{\gamma}\right)\boldsymbol{\nu}^{(i)}_p\left(\boldsymbol{\gamma}\right)^{\top}\boldsymbol{e}_{j+1}\right)^2}{\tau^{(i,p)}_{k-1}\left(\boldsymbol{\gamma}\right)} = {\boldsymbol{\nu}^{(i)}_p\left(\boldsymbol{\gamma}\right)^\top\boldsymbol{\sigma\left(\boldsymbol{\gamma}\right)}\boldsymbol{\nu}^{(i)}_p\left(\boldsymbol{\gamma}\right)}. 
\end{align}
From \eqref{eq:eigen-mg}, \eqref{eq:c}, \eqref{eq:lil2-c} and \eqref{eq:qsl2-b}, \eqref{eq:qsl2-a} follows, establishing \eqref{eq:qsl2}.
\end{proof}

\begin{lemma}\label{lem:3}
    Fix $i \in \{1,2\}$. Let $p \neq 7/8$ be such that $- \frac{1}{2} < \lambda^{(i)}_p\left(\boldsymbol{\Gamma}^{(p)}\right) < 0$. Then, there exists a finite random variable $\boldsymbol{M}^{(i,p)}$ such that 
    \begin{align}\label{eq:clt3}  &\frac{{{\boldsymbol{\nu}^{(i)}_p\left(\boldsymbol{\Gamma}^{(p)}\right)^\top}\boldsymbol{M}_{n+1}\left(\boldsymbol{\Gamma}^{(p)}\right)}}{{{n}^{\lambda^{(i)}_p\left(\boldsymbol{\Gamma}^{(p)}\right)}}} \stackrel{a.s.}{\to} \boldsymbol{M}^{(i,p)}. 
    \end{align}
\end{lemma}

\begin{proof}
To show \eqref{eq:clt3}, it is enough show, for any $\boldsymbol{\gamma} \in \operatorname{supp}\left(\boldsymbol{\Gamma}^{(p)}\right)$, that on the event $\left\{\boldsymbol{\Gamma}^{(p)} = \boldsymbol{\gamma}\right\}$, there exists a finite random variable $\boldsymbol{M}^{(i,p)}\left(\boldsymbol{\gamma}\right)$ such that 
\begin{align}\label{eq:clt3-a}  
&\frac{{{\boldsymbol{\nu}^{(i)}_p\left(\boldsymbol{\gamma}\right)^\top}\boldsymbol{M}_{kn+1}\left(\boldsymbol{\gamma}\right)}}{{{n}^{\lambda^{(i)}_p\left(\boldsymbol{\gamma}\right)}}} \stackrel{a.s.}{\to} \boldsymbol{M}^{(i,p)}\left(\boldsymbol{\gamma}\right). 
\end{align}
From \eqref{eq:eigen-mg} and Lemma~\ref{lem:sa}, it follows that $\left(e^{-a_n\boldsymbol{\lambda}^{(i)}_p\left(\boldsymbol{\gamma}\right)}{\boldsymbol{\nu}^{(i)}_p\left(\boldsymbol{\gamma}\right)^\top}\boldsymbol{M}_{n+1}\left(\boldsymbol{\gamma}\right)\right)_{n \geq 2}$ is a martingale and we have $e^{-a_n\boldsymbol{\lambda}^{(i)}_p\left(\boldsymbol{\gamma}\right)}{\boldsymbol{\nu}^{(i)}_p\left(\boldsymbol{\gamma}\right)^\top}\boldsymbol{M}_{n+1}\left(\boldsymbol{\gamma}\right) = \sum_{k=2}^n\Phi^{(i,p)}_{k}\left(\boldsymbol{\gamma}\right)\boldsymbol{\nu}^{(i)}_p\left(\boldsymbol{\gamma}\right)^{\top}\boldsymbol{e}_{k+1}$ with 
\begin{align*}
   \lim_{n\to \infty} \sum_{k=2}^n\left|\Phi^{(i,p)}_{k}\left(\boldsymbol{\gamma}\right)\right|^2 < \infty.
\end{align*}
Hence Theorem 1.3.24 of \cite{duflo2013random} gives us that for some finite random variable $\boldsymbol{M}^{(i,p)}\left(\boldsymbol{\gamma}\right)$, 
\begin{align}\label{eq:clt3-b}
    e^{-a_n\boldsymbol{\lambda}^{(i)}_p\left(\boldsymbol{\gamma}\right)}{\boldsymbol{\nu}^{(i)}_p\left(\boldsymbol{\gamma}\right)^\top}\boldsymbol{M}_{n+1}\left(\boldsymbol{\gamma}\right) \stackrel{a.s.}{\to} \frac{\boldsymbol{M}^{(i,p)}\left(\boldsymbol{\gamma}\right)}{c_{i,p}\left(\boldsymbol{\gamma}\right)}.
\end{align}
From \eqref{eq:clt3-b} and \eqref{eq:c}, we get \eqref{eq:clt3-a}, establishing \eqref{eq:clt3}.
\end{proof}

\begin{lemma}\label{lem:4}
     Fix $i \in \{1,2\}$. Let $p \neq 7/8$ be such that $- \frac{1}{2} \leq \lambda^{(1)}_p\left(\boldsymbol{\Gamma}^{(p)}\right) < 0$. Then there exists a finite random variable $\boldsymbol{\Delta}^{(i,p)}$ such that \begin{align}\label{eq:delta1}\frac{{\boldsymbol{\nu}^{(i)}_p\left(\boldsymbol{\Gamma}^{(p)}\right)^\top}\boldsymbol{\Delta}_{n+1}\left(\boldsymbol{\Gamma}^{(p)}\right)}{n^{\lambda^{(1)}_p\left(\boldsymbol{\Gamma}^{(p)}\right)}} \stackrel{a.s.}{\to} \boldsymbol{\Delta}^{(i,p)}.
\end{align}
\end{lemma}

\begin{proof}
We first show that $\|\boldsymbol{\Delta}_{n+1}\left(\boldsymbol{\Gamma}_p\right)\| \to 0$ almost surely. 
Note that 
\begin{align*}
    \|\boldsymbol{\Delta}_{n+1}\left(\boldsymbol{\Gamma}_p\right)\| \leq \|\boldsymbol{\Gamma}_{n+1} - \boldsymbol{\Gamma}_p\| + \|\boldsymbol{M}_{n+1}\left(\boldsymbol{\Gamma}_p\right)\|.
\end{align*}
Theorem~\ref{thm:sa-slln} implies that $\|\boldsymbol{\Gamma}_{n+1} - \boldsymbol{\Gamma}_p\| \to 0$ almost surely. Also, from \eqref{eq:lil1}, we get, for all $\epsilon > 0$, 
that, if $ {\lambda}^{(2)}_p\left(\boldsymbol{\Gamma}^{(p)}\right) < -1/2$ almost surely, then 
\begin{align*}
{n}^{1/2-\epsilon}\left\{{\boldsymbol{\nu}^{(2)}_p\left(\boldsymbol{\Gamma}^{(p)}\right)^\top}\boldsymbol{M}_{n+1}\left(\boldsymbol{\Gamma}^{(p)}\right)\right\} \stackrel{a.s.}{\to} 0,
\end{align*}
and from \eqref{eq:lil2} and \eqref{eq:clt3}, we get, for all $\epsilon > 0$ and $i = 1,2$, that, if $ -1/2 \leq {\lambda}^{(i)}_p\left(\boldsymbol{\Gamma}^{(p)}\right) < 0$ almost surely, then 
\begin{align*}
{n}^{-{\lambda}^{(i)}_p\left(\boldsymbol{\Gamma}^{(p)}\right)-\epsilon}\left\{{\boldsymbol{\nu}^{(i)}_p\left(\boldsymbol{\Gamma}^{(p)}\right)^\top}\boldsymbol{M}_{n+1}\left(\boldsymbol{\Gamma}^{(p)}\right)\right\} \stackrel{a.s.}{\to} 0.
\end{align*}
Since $\boldsymbol{\nu}^{(1)}_p\left(\boldsymbol{\Gamma}^{(p)}\right)$ and $\boldsymbol{\nu}^{(2)}_p\left(\boldsymbol{\Gamma}^{(p)}\right)$ are linearly independent and \[\|\boldsymbol{M}_{n+1}\left(\boldsymbol{\Gamma}_p\right)\|^2 
= \left(\begin{pmatrix}1 & 0\end{pmatrix}^{\top}\boldsymbol{M}_{n+1}\left(\boldsymbol{\Gamma}_p\right)\right)^2 + \left(\begin{pmatrix}0 & 1\end{pmatrix}^{\top}\boldsymbol{M}_{n+1}\left(\boldsymbol{\Gamma}_p\right)\right)^2,\]we conclude, for all $\epsilon > 0$, that, 
\begin{align}\label{eq:delta7}n^{-{\lambda}^{(1)}_p\left(\boldsymbol{\Gamma}^{(p)}\right)-\epsilon}\|\boldsymbol{M}_{n+1}\left(\boldsymbol{\Gamma}_p\right)\| \stackrel{a.s.}{\to} 0.\end{align}  Consequently, $\|\boldsymbol{\Delta}_{n+1}\left(\boldsymbol{\Gamma}_p\right)\|^2 \to 0$ almost surely. Next, define, for $\boldsymbol{\gamma} \in \operatorname{supp}\left(\boldsymbol{\Gamma}^{(p)}\right)$,
\begin{align}\label{eq:delta3}
r_{n+1}\left(\boldsymbol{\gamma}\right) = {\boldsymbol{h}_p\left(\boldsymbol{\Gamma}_n\right) - \boldsymbol{J}_p\left(\boldsymbol{\gamma}\right)\left(\boldsymbol{\Gamma}_{n} - \boldsymbol{\gamma}\right)}, \quad \rho_{n+1}\left(\boldsymbol{\gamma}\right) = \left(e^{\frac{\boldsymbol{J}_p\left(\boldsymbol{\gamma}\right)}{n+1}} - \boldsymbol{I} - \frac{\boldsymbol{J}_p\left(\boldsymbol{\gamma}\right)}{n+1}\right)\boldsymbol{M}_n\left(\boldsymbol{\gamma}\right).\end{align} 
Note that, there exists a random variable $0 < C < \infty$  
such that almost surely
\begin{align}\label{eq:delta4}
    \|r_{n+1}\left(\boldsymbol{\Gamma}_p\right)\| &\leq C\|\boldsymbol{\Gamma}_n - \boldsymbol{\Gamma}_p\|^2
    \leq 2C\left(\|\boldsymbol{M}_n\left(\boldsymbol{\Gamma}_p\right)\|^2 + \|\boldsymbol{\Delta}_{n}\left(\boldsymbol{\Gamma}_p\right)\|^2\right).
\end{align}
Fix $\delta > 0$. As $2C\left(\boldsymbol{\Gamma}_p\right)\|\boldsymbol{\Delta}_{n}\left(\boldsymbol{\Gamma}_p\right)\| \to 0$ almost surely, there exists an positive integer valued random variable $N_0$ such that, almost surely, for all $n \geq N_0$, 
\begin{align}\label{eq:delta8}
    \|r_{n+1}\left(\boldsymbol{\Gamma}_p\right)\|
    \leq 2C\|\boldsymbol{M}_n\left(\boldsymbol{\Gamma}_p\right)\|^2 + \delta\|\boldsymbol{\Delta}_{n}\left(\boldsymbol{\Gamma}_p\right)\|.
\end{align}
Also, observe that, there exists a random variable $0 < D < \infty$ such that almost surely
\begin{align}\label{eq:delta5}
    \|\rho_{n+1}\left(\boldsymbol{\Gamma}_p\right)\| \leq \frac{D\|M_{n}\left(\boldsymbol{\Gamma}_p\right)\|}{(n+1)^2}.
\end{align}
From \eqref{eq:not1} and \eqref{eq:delta3}, we have 
\begin{align}\label{eq:delta9}
    \boldsymbol{\Delta}_{n+1}\left(\boldsymbol{\Gamma}_p\right) &= \boldsymbol{\Gamma}_{n+1} - \boldsymbol{\Gamma}_p - \boldsymbol{M}_{n+1}\left(\boldsymbol{\Gamma}_p\right) \nonumber \\ &= \boldsymbol{\Gamma}_{n} - \boldsymbol{\Gamma}_p + \frac{\boldsymbol{h}_p\left(\boldsymbol{\Gamma_n}\right)}{n+1} + \frac{\boldsymbol{e}_{n+1}}{n+1} - \frac{\boldsymbol{e}_{n+1}}{n+1} - e^{\frac{\boldsymbol{J}_p\left(\boldsymbol{\Gamma}_p\right)}{n+1}}\boldsymbol{M}_n\left(\boldsymbol{\Gamma}_p\right)
    \nonumber \\ &= \left(1+\frac{\boldsymbol{J}_p\left(\boldsymbol{\Gamma}_p\right)}{n+1}\right) \boldsymbol{\Delta}_{n}\left(\boldsymbol{\Gamma}_p\right) + \frac{r_{n+1}\left(\boldsymbol{\Gamma}_p\right)}{n+1} + \rho_{n+1}\left(\boldsymbol{\Gamma}_p\right),
\end{align}
so that, \eqref{eq:delta8} and \eqref{eq:delta5} imply, almost surely, for all $n \geq N_0$, 
\begin{align*}
    \|\boldsymbol{\Delta}_{n+1}\left(\boldsymbol{\Gamma}_p\right)\| &\leq \left(1+\frac{\lambda^{(1)}_p\left(\boldsymbol{\Gamma}_p\right)}{n+1}\right) \|\boldsymbol{\Delta}_{n}\left(\boldsymbol{\Gamma}_p\right)\| + \frac{\|r_{n+1}\left(\boldsymbol{\Gamma}_p\right)\|}{n+1} + \|\rho_{n+1}\left(\boldsymbol{\Gamma}_p\right)\|
    \\ &\leq \left(1+\frac{\lambda^{(1)}_p\left(\boldsymbol{\Gamma}_p\right)+\delta}{n+1}\right) \|\boldsymbol{\Delta}_{n}\left(\boldsymbol{\Gamma}_p\right)\| + \frac{2C\|M_{n}\left(\boldsymbol{\Gamma}_p\right)\|^2}{n+1} + \frac{D\|M_{n}\left(\boldsymbol{\Gamma}_p\right)\|}{(n+1)^2}\\
    &\leq \exp\left\{\frac{\lambda^{(1)}_p\left(\boldsymbol{\Gamma}_p\right)+\delta}{n+1}\right\}\|\boldsymbol{\Delta}_{n}\left(\boldsymbol{\Gamma}_p\right)\| + \frac{2C\|M_{n}\left(\boldsymbol{\Gamma}_p\right)\|^2}{n+1} + \frac{D\|M_{n}\left(\boldsymbol{\Gamma}_p\right)\|}{(n+1)^2}
    \\&\leq \exp\left\{\sum_{k=N_0}^{n}\frac{\lambda^{(1)}_p\left(\boldsymbol{\Gamma}_p\right)+\delta}{k+1}\right\}\left[\|\boldsymbol{\Delta}_{N_0}\left(\boldsymbol{\Gamma}_p\right)\|\right.\\ 
    &\hspace{1.5cm}\left.+\sum_{j = N_0}^n\exp\left\{-\sum_{k=N_0}^{j}\frac{\lambda^{(1)}_p\left(\boldsymbol{\Gamma}_p\right)+\delta}{k+1}\right\}\left(\frac{2C\|M_{j}\left(\boldsymbol{\Gamma}_p\right)\|^2}{j+1} + \frac{D\|M_{j}\left(\boldsymbol{\Gamma}_p\right)\|}{(j+1)^2}\right)\right].
\end{align*}
Using 
\eqref{eq:delta7} (with $\epsilon < 
1- \delta$), the facts that almost surely ${\lambda}^{(1)}_p\left(\boldsymbol{\Gamma}^{(p)}\right) \geq -1$ and 
\[
   \exp\left\{\sum_{k=N_0}^{n}\frac{\lambda^{(1)}_p\left(\boldsymbol{\Gamma}_p\right)+\delta}{k+1}\right\} = O_p\left(n^{\lambda^{(1)}_p\left(\boldsymbol{\Gamma}_p\right)+\delta}\right), \quad \text{ as } j \to \infty,
\]
we get that, almost surely, 
$\|\boldsymbol{\Delta}_{n+1}\left(\boldsymbol{\Gamma}_p\right)\| = O_p\left(n^{\lambda^{(1)}_p\left(\boldsymbol{\Gamma}_p\right)+\delta}\right)$ as $n \to \infty$. Consequently, \eqref{eq:delta4} 
and \eqref{eq:delta7} (with $\epsilon = \delta$) imply that, 
almost surely
\begin{align}\label{eq:delta10}
     \|r_{n+1}\left(\boldsymbol{\Gamma}_p\right)\| = O_p\left(n^{2\lambda^{(1)}_p\left(\boldsymbol{\Gamma}_p\right)+2\delta}\right), \quad n \to \infty.
\end{align}
From \eqref{eq:delta9}, we get, for any $N \geq 1$,
\begin{align*}
    &{\boldsymbol{\nu}^{(i)}_p\left(\boldsymbol{\Gamma}^{(p)}\right)^\top}\boldsymbol{\Delta}_{n+1}\left(\boldsymbol{\Gamma}_p\right) = \left(1+\frac{{{\lambda}^{(i)}_p\left(\boldsymbol{\Gamma}^{(p)}\right)}}{n+1}\right) {\boldsymbol{\nu}^{(i)}_p\left(\boldsymbol{\Gamma}^{(p)}\right)^\top}\boldsymbol{\Delta}_{n}\left(\boldsymbol{\Gamma}_p\right) \\ & \hspace{5 cm} + \frac{{\boldsymbol{\nu}^{(i)}_p\left(\boldsymbol{\Gamma}^{(p)}\right)^\top}r_{n+1}\left(\boldsymbol{\Gamma}_p\right)}{n+1} + {\boldsymbol{\nu}^{(i)}_p\left(\boldsymbol{\Gamma}^{(p)}\right)^\top}\rho_{n+1}\left(\boldsymbol{\Gamma}_p\right)
    \\ &= \prod_{k=N}^{n}\left(1+\frac{{{\lambda}^{(i)}_p\left(\boldsymbol{\Gamma}^{(p)}\right)}}{k+1}\right) \left[{\boldsymbol{\nu}^{(i)}_p\left(\boldsymbol{\Gamma}^{(p)}\right)^\top}\boldsymbol{\Delta}_{n}\left(\boldsymbol{\Gamma}_p\right)\right.\\ 
    &\hspace{1.5cm}\left.+\sum_{j = N}^n\prod_{k=N}^{j}\left(1+\frac{{{\lambda}^{(i)}_p\left(\boldsymbol{\Gamma}^{(p)}\right)}}{k+1}\right)^{-1}\left(\frac{{\boldsymbol{\nu}^{(i)}_p\left(\boldsymbol{\Gamma}^{(p)}\right)^\top}r_{k+1}\left(\boldsymbol{\Gamma}_p\right)}{k+1} + {\boldsymbol{\nu}^{(i)}_p\left(\boldsymbol{\Gamma}^{(p)}\right)^\top}\rho_{k+1}\left(\boldsymbol{\Gamma}_p\right)\right)\right].
\end{align*}
Observe that, \eqref{eq:delta10} implies that, almost surely
\[ 
    \frac{{\boldsymbol{\nu}^{(i)}_p\left(\boldsymbol{\Gamma}^{(p)}\right)^\top}r_{n+1}\left(\boldsymbol{\Gamma}_p\right)}{n+1} = O_p\left(n^{2\lambda^{(1)}_p\left(\boldsymbol{\Gamma}_p\right)+2\delta-1}\right), \quad n \to \infty,
\]
and, \eqref{eq:delta7}, \eqref{eq:delta5} imply that, almost surely
\[ 
    {\boldsymbol{\nu}^{(i)}_p\left(\boldsymbol{\Gamma}^{(p)}\right)^\top}\rho_{n+1}\left(\boldsymbol{\Gamma}_p\right) = O_p\left(n^{\lambda^{(1)}_p\left(\boldsymbol{\Gamma}_p\right)+2\delta-2}\right), \quad n \to \infty.
\]
Note that ${\lambda^{(1)}_p\left(\boldsymbol{\Gamma}_p\right)+2\delta-2} \leq {2\lambda^{(1)}_p\left(\boldsymbol{\Gamma}_p\right)+2\delta-1}$ and almost surely
\[
  \lim_{n \to \infty}\frac{1}{n^{{{\lambda}^{(i)}_p\left(\boldsymbol{\Gamma}^{(p)}\right)}}}\prod_{k=N}^{n}\left(1+\frac{{{\lambda}^{(i)}_p\left(\boldsymbol{\Gamma}^{(p)}\right)}}{k+1}\right) < \infty.
\]
Now consider the case $i = 2$. Note that we can always choose $\delta$ such that almost surely $2\lambda^{(1)}_p\left(\boldsymbol{\Gamma}_p\right) - \lambda^{(2)}_p\left(\boldsymbol{\Gamma}_p\right)+2\delta \neq 0$. If $2\lambda^{(1)}_p\left(\boldsymbol{\Gamma}_p\right) - \lambda^{(2)}_p\left(\boldsymbol{\Gamma}_p\right)+2\delta > 0$ almost surely, then for all sufficiently large $N$, we have, almost surely,
\begin{align*}
\sum_{j = N}^n\prod_{k=N}^{j}\left(1+\frac{{{\lambda}^{(2)}_p\left(\boldsymbol{\Gamma}^{(p)}\right)}}{k+1}\right)^{-1}\left(\frac{{\boldsymbol{\nu}^{(2)}_p\left(\boldsymbol{\Gamma}^{(p)}\right)^\top}r_{k+1}\left(\boldsymbol{\Gamma}_p\right)}{k+1} + {\boldsymbol{\nu}^{(2)}_p\left(\boldsymbol{\Gamma}^{(p)}\right)^\top}\rho_{k+1}\left(\boldsymbol{\Gamma}_p\right)\right) \\ = O_p\left(n^{2\lambda^{(1)}_p\left(\boldsymbol{\Gamma}_p\right) - \lambda^{(2)}_p\left(\boldsymbol{\Gamma}_p\right)+2\delta}\right), \quad n \to \infty,\end{align*}
and so, in this case, for sufficiently small $\delta$,
\begin{align*}
     &\lim_{n \to \infty}\frac{{\boldsymbol{\nu}^{(2)}_p\left(\boldsymbol{\Gamma}^{(p)}\right)^\top}\boldsymbol{\Delta}_{n+1}\left(\boldsymbol{\Gamma}^{(p)}\right)}{n^{{{\lambda}^{(1)}_p\left(\boldsymbol{\Gamma}^{(p)}\right)}}}\\ &\hspace{2cm}\stackrel{a.s.}{=} \lim_{n \to \infty}\frac{1}{n^{{{\lambda}^{(2)}_p\left(\boldsymbol{\Gamma}^{(p)}\right)}}}\prod_{k=N}^{n}\left(1+\frac{{{\lambda}^{(2)}_p\left(\boldsymbol{\Gamma}^{(p)}\right)}}{k+1}\right)O_p\left({n^{{{\lambda}^{(1)}_p\left(\boldsymbol{\Gamma}^{(p)}\right)}+\delta}}\right)
     \stackrel{a.s.}{=} 0 =: \boldsymbol{\Delta}^{(2,p)}.
\end{align*}
If almost surely $2\lambda^{(1)}_p\left(\boldsymbol{\Gamma}_p\right) - \lambda^{(2)}_p\left(\boldsymbol{\Gamma}_p\right)+2\delta < 0$, for all sufficiently large $N$, almost surely,
\begin{align*}
\lim_{n\to\infty}\sum_{j = N}^n\prod_{k=N}^{j}\left(1+\frac{{{\lambda}^{(2)}_p\left(\boldsymbol{\Gamma}^{(p)}\right)}}{k+1}\right)^{-1}\left(\frac{{\boldsymbol{\nu}^{(2)}_p\left(\boldsymbol{\Gamma}^{(p)}\right)^\top}r_{k+1}\left(\boldsymbol{\Gamma}_p\right)}{k+1} + {\boldsymbol{\nu}^{(2)}_p\left(\boldsymbol{\Gamma}^{(p)}\right)^\top}\rho_{k+1}\left(\boldsymbol{\Gamma}_p\right)\right) < \infty,\end{align*}
and so, in this case, 
\begin{align*}
     &\lim_{n \to \infty}\frac{{\boldsymbol{\nu}^{(2)}_p\left(\boldsymbol{\Gamma}^{(p)}\right)^\top}\boldsymbol{\Delta}_{n+1}\left(\boldsymbol{\Gamma}^{(p)}\right)}{n^{{{\lambda}^{(1)}_p\left(\boldsymbol{\Gamma}^{(p)}\right)}}}\\ &\stackrel{a.s.}{=} \lim_{n \to \infty}\frac{1}{n^{{{\lambda}^{(2)}_p\left(\boldsymbol{\Gamma}^{(p)}\right)}}}\prod_{k=N}^{n}\left(1+\frac{{{\lambda}^{(2)}_p\left(\boldsymbol{\Gamma}^{(p)}\right)}}{k+1}\right){n^{{{\lambda}^{(1)}_p\left(\boldsymbol{\Gamma}^{(p)}\right)}-{\lambda}^{(2)}_p\left(\boldsymbol{\Gamma}^{(p)}\right)}}
     \\ &\times\sum_{j = N}^n\prod_{k=N}^{j}\left(1+\frac{{{\lambda}^{(2)}_p\left(\boldsymbol{\Gamma}^{(p)}\right)}}{k+1}\right)^{-1}\left(\frac{{\boldsymbol{\nu}^{(2)}_p\left(\boldsymbol{\Gamma}^{(p)}\right)^\top}r_{k+1}\left(\boldsymbol{\Gamma}_p\right)}{k+1} + {\boldsymbol{\nu}^{(2)}_p\left(\boldsymbol{\Gamma}^{(p)}\right)^\top}\rho_{k+1}\left(\boldsymbol{\Gamma}_p\right)\right)\stackrel{a.s.}{=} \boldsymbol{\Delta}^{(2,p)}.
\end{align*}
Finally consider the case $i = 1$. Here we can always choose $\delta$ such that $\lambda^{(1)}_p\left(\boldsymbol{\Gamma}_p\right) +2\delta < 0$. Consequently, for all sufficiently large $N$,
\begin{align*}
\lim_{n\to\infty}\sum_{j = N}^n\prod_{k=N}^{j}\left(1+\frac{{{\lambda}^{(1)}_p\left(\boldsymbol{\Gamma}^{(p)}\right)}}{k+1}\right)^{-1}\left(\frac{{\boldsymbol{\nu}^{(1)}_p\left(\boldsymbol{\Gamma}^{(p)}\right)^\top}r_{k+1}\left(\boldsymbol{\Gamma}_p\right)}{k+1} + {\boldsymbol{\nu}^{(1)}_p\left(\boldsymbol{\Gamma}^{(p)}\right)^\top}\rho_{k+1}\left(\boldsymbol{\Gamma}_p\right)\right) < \infty,\end{align*}
and so, in this case, 
\begin{align*}
     &\lim_{n \to \infty}\frac{{\boldsymbol{\nu}^{(1)}_p\left(\boldsymbol{\Gamma}^{(p)}\right)^\top}\boldsymbol{\Delta}_{n+1}\left(\boldsymbol{\Gamma}^{(p)}\right)}{n^{{{\lambda}^{(1)}_p\left(\boldsymbol{\Gamma}^{(p)}\right)}}}\\ &\stackrel{a.s.}{=} \lim_{n \to \infty}\frac{1}{n^{{{\lambda}^{(1)}_p\left(\boldsymbol{\Gamma}^{(p)}\right)}}}\prod_{k=N}^{n}\left(1+\frac{{{\lambda}^{(1)}_p\left(\boldsymbol{\Gamma}^{(p)}\right)}}{k+1}\right) \\ &\times \sum_{j = N}^n\prod_{k=N}^{j}\left(1+\frac{{{\lambda}^{(1)}_p\left(\boldsymbol{\Gamma}^{(p)}\right)}}{k+1}\right)^{-1}\left(\frac{{\boldsymbol{\nu}^{(1)}_p\left(\boldsymbol{\Gamma}^{(p)}\right)^\top}r_{k+1}\left(\boldsymbol{\Gamma}_p\right)}{k+1} + {\boldsymbol{\nu}^{(1)}_p\left(\boldsymbol{\Gamma}^{(p)}\right)^\top}\rho_{k+1}\left(\boldsymbol{\Gamma}_p\right)\right)
     =: \boldsymbol{\Delta}^{(1,p)}.
\end{align*}
This completes the proof.
\end{proof}

We get the following results by combining Lemma~\eqref{lem:1}, Lemma~\eqref{lem:2}, Lemma~\eqref{lem:3} and Lemma~\eqref{lem:4}.

\begin{theorem}
Let $c_1, c_2 \in \mathbb{R}$ and $p \neq 7/8$ be such that $\lambda^{(2)}_p\left(\boldsymbol{\Gamma}^{(p)}\right) < \lambda^{(1)}_p\left(\boldsymbol{\Gamma}^{(p)}\right) = - \frac{1}{2}$ almost surely. 
 \begin{itemize}
    \item[i)] The following central limit theorem holds - 
    \begin{align}\label{eq:mclt2}  &\sqrt{\frac{n}{\log n}}\left\{\left(c_1\boldsymbol{\nu}^{(1)}_p\left(\boldsymbol{\Gamma}^{(p)}\right) + c_2\boldsymbol{\nu}^{(2)}_p\left(\boldsymbol{\Gamma}^{(p)}\right)\right)^\top\left(\boldsymbol{\Gamma}_n - \boldsymbol{\Gamma}^{(p)}\right)\right\} \stackrel{d}{\to} N\left(0, {\Sigma}^{(2)}_{c_1,c_2,p}\right),
    \end{align}
    where
    \begin{align}\label{eq:sig2}
        {\Sigma}^{(2)}_{c_1,c_2,p} = c^2_1\boldsymbol{\nu}^{(1)}_p\left(\boldsymbol{\Gamma}^{(p)}\right)^{\top}\sigma\left(\boldsymbol{\Gamma}^{(p)}\right)\boldsymbol{\nu}^{(1)}_p\left(\boldsymbol{\Gamma}^{(p)}\right).
    \end{align}
    \item[ii)] The following law of iterated logarithms holds almost surely -  
    \begin{align}\label{eq:mlil2}  &\limsup_{n \to \infty}\sqrt{\frac{n}{2\log n\log\log\log n}}\left\{\left(c_1\boldsymbol{\nu}^{(1)}_p\left(\boldsymbol{\Gamma}^{(p)}\right) + c_2\boldsymbol{\nu}^{(2)}_p\left(\boldsymbol{\Gamma}^{(p)}\right)\right)^\top\left(\boldsymbol{\Gamma}_n - \boldsymbol{\Gamma}^{(p)}\right)\right\} \\ = - &\liminf_{n \to \infty}\sqrt{\frac{n}{2\log n\log\log\log n}}\left\{\left(c_1\boldsymbol{\nu}^{(2)}_p\left(\boldsymbol{\Gamma}^{(p)}\right) + c_2\boldsymbol{\nu}^{(2)}_p\left(\boldsymbol{\Gamma}^{(p)}\right)\right)^\top\left(\boldsymbol{\Gamma}_n - \boldsymbol{\Gamma}^{(p)}\right)\right\} \nonumber \\ =\textcolor{white}{+}&\sqrt{{\Sigma}^{(2)}_{c_1,c_2,p}} . \nonumber 
    \end{align}
    \item[iii)] The following quadratic strong law holds almost surely -  
    \begin{align}\label{eq:mqsl2}  &\lim_{n \to \infty} \frac{1}{\log \log n}\sum_{k=2}^n\left(\frac{1}{\log k}\right)^2\left\{\left(c_1\boldsymbol{\nu}^{(1)}_p\left(\boldsymbol{\Gamma}^{(p)}\right) + c_2\boldsymbol{\nu}^{(2)}_p\left(\boldsymbol{\Gamma}^{(p)}\right)\right)^\top\left(\boldsymbol{\Gamma}_k - \boldsymbol{\Gamma}^{(p)}\right)\right.\\ &\hspace{3cm}\times \left.\left(\boldsymbol{\Gamma}_k - \boldsymbol{\Gamma}^{(p)}\right)^{\top}\left(c_1\boldsymbol{\nu}^{(1)}_p\left(\boldsymbol{\Gamma}^{(p)}\right) + c_2\boldsymbol{\nu}^{(2)}_p\left(\boldsymbol{\Gamma}^{(p)}\right)\right)\right\} = {{\Sigma}^{(2)}_{c_1,c_2,p}}. \nonumber
    \end{align}
    \end{itemize}    
\end{theorem}

\begin{proof}
As $\lambda^{(2)}_p\left(\boldsymbol{\Gamma}^{(p)}\right) < \lambda^{(1)}_p\left(\boldsymbol{\Gamma}^{(p)}\right) = - \frac{1}{2}$ almost surely, from \eqref{eq:not1}, \eqref{eq:clt1}, \eqref{eq:clt2} and \eqref{eq:delta1}, \eqref{eq:mclt2} follows. Similarly, from \eqref{eq:not1}, \eqref{eq:lil1}, \eqref{eq:lil2} and \eqref{eq:delta1}, \eqref{eq:mlil2} follows.
To show \eqref{eq:mqsl2}, we show that 
\begin{align}\label{q1}
    &\lim_{n \to \infty} \frac{1}{\log \log n}\sum_{k=2}^n\left(\frac{1}{\log k}\right)^2\left\{\left(c_1\boldsymbol{\nu}^{(1)}_p\left(\boldsymbol{\Gamma}^{(p)}\right) + c_2\boldsymbol{\nu}^{(2)}_p\left(\boldsymbol{\Gamma}^{(p)}\right)\right)^\top\boldsymbol{M}_k\right.\\ &\hspace{3cm}\times \left.\boldsymbol{M}_k^{\top}\left(c_1\boldsymbol{\nu}^{(1)}_p\left(\boldsymbol{\Gamma}^{(p)}\right) + c_2\boldsymbol{\nu}^{(2)}_p\left(\boldsymbol{\Gamma}^{(p)}\right)\right)\right\} = {{\Sigma}^{(2)}_{c_1,c_2,p}}. \nonumber
\end{align}
\begin{align}\label{q2}
    &\lim_{n \to \infty} \frac{1}{\log \log n}\sum_{k=2}^n\left(\frac{1}{\log k}\right)^2\left\{\left(c_1\boldsymbol{\nu}^{(1)}_p\left(\boldsymbol{\Gamma}^{(p)}\right) + c_2\boldsymbol{\nu}^{(2)}_p\left(\boldsymbol{\Gamma}^{(p)}\right)\right)^\top\boldsymbol{M}_k\right.\\ &\hspace{3cm}\times \left.\boldsymbol{\Delta}_k^{\top}\left(c_1\boldsymbol{\nu}^{(1)}_p\left(\boldsymbol{\Gamma}^{(p)}\right) + c_2\boldsymbol{\nu}^{(2)}_p\left(\boldsymbol{\Gamma}^{(p)}\right)\right)\right\} \nonumber
    \\ =&\lim_{n \to \infty} \frac{1}{\log \log n}\sum_{k=2}^n\left(\frac{1}{\log k}\right)^2\left\{\left(c_1\boldsymbol{\nu}^{(1)}_p\left(\boldsymbol{\Gamma}^{(p)}\right) + c_2\boldsymbol{\nu}^{(2)}_p\left(\boldsymbol{\Gamma}^{(p)}\right)\right)^\top\boldsymbol{\Delta}_k\right. \nonumber\\ &\hspace{3cm}\times \left.\boldsymbol{\Delta}_k^{\top}\left(c_1\boldsymbol{\nu}^{(1)}_p\left(\boldsymbol{\Gamma}^{(p)}\right) + c_2\boldsymbol{\nu}^{(2)}_p\left(\boldsymbol{\Gamma}^{(p)}\right)\right)\right\} = 0. \nonumber
\end{align}
Since $\boldsymbol{\nu}^{(1)}_p\left(\boldsymbol{\Gamma}^{(p)}\right)$ and $\boldsymbol{\nu}^{(2)}_p\left(\boldsymbol{\Gamma}^{(p)}\right)$ are linearly independent and \[\|\boldsymbol{x}\|^2 
= \left(\begin{pmatrix}1 & 0\end{pmatrix}^{\top}\boldsymbol{x}\right)^2 + \left(\begin{pmatrix}0 & 1\end{pmatrix}^{\top}\boldsymbol{x}\right)^2,\] from \eqref{eq:delta1}, \eqref{eq:lil1} and  \eqref{eq:lil2}, it follows that \[\|\boldsymbol{\Delta}_n\|^2 = O_p\left(\frac{1}{n}\right), \quad \|\boldsymbol{\Delta}_n\|\|\boldsymbol{M}_n\| = O_p\left(\frac{\sqrt{\log n \log \log \log n}{}}{n}\right),\]
implying \eqref{q2}. From \eqref{eq:lil1} and  \eqref{eq:lil2}, we get, for $(i_1, i_2) \in \{(1,2), (2,1) , (2,2)\}$, that 
\begin{align}\label{q3}
    &\lim_{n \to \infty} \frac{1}{\log \log n}\sum_{k=2}^n\left(\frac{1}{\log k}\right)^2\left\{\left(\boldsymbol{\nu}^{(i_1)}_p\left(\boldsymbol{\Gamma}^{(p)}\right)\right)^\top\boldsymbol{M}_k\boldsymbol{M}_k^{\top}\left(\boldsymbol{\nu}^{(i_2)}_p\left(\boldsymbol{\Gamma}^{(p)}\right)\right)\right\} \\ =& O_p\left(\frac{1}{\log \log n}\sum_{k=2}^n\frac{1}{k \log k}\left(\frac{\log \log k \log \log \log k}{\log k}\right)^{1/2}\right) = o_p(1). 
\end{align}
From \eqref{q3} and \eqref{eq:qsl2}, we get \eqref{q1}, implying \eqref{eq:mqsl2}.
\end{proof}

\begin{theorem}
Let $c_1, c_2 \in \mathbb{R}$ and $p \neq 7/8$ be such that $\lambda^{(2)}_p\left(\boldsymbol{\Gamma}^{(p)}\right) < - \frac{1}{2} < \lambda^{(1)}_p\left(\boldsymbol{\Gamma}^{(p)}\right) < 0$ almost surely. Then, there exists a finite random variable $\boldsymbol{L}^{(p)}_{c_1,c_2}$ such that 
\begin{align}\label{eq:mclt3}  &\frac{\left(c_1\boldsymbol{\nu}^{(1)}_p\left(\boldsymbol{\Gamma}^{(p)}\right) + c_2\boldsymbol{\nu}^{(2)}_p\left(\boldsymbol{\Gamma}^{(p)}\right)\right)^\top\left(\boldsymbol{\Gamma}_n - \boldsymbol{\Gamma}^{(p)}\right)}{{{n}^{\lambda^{(1)}_p\left(\boldsymbol{\Gamma}^{(p)}\right)}}} \stackrel{a.s.}{\to} \boldsymbol{L}^{(p)}_{c_1,c_2}. 
\end{align}
\end{theorem}
    
\begin{proof}
    As $\lambda^{(2)}_p\left(\boldsymbol{\Gamma}^{(p)}\right) < - \frac{1}{2} < \lambda^{(1)}_p\left(\boldsymbol{\Gamma}^{(p)}\right) < 0$ almost surely, from \eqref{eq:not1}, \eqref{eq:clt3}, \eqref{eq:lil1} and \eqref{eq:delta1}, \eqref{eq:mclt3} follows, with $\boldsymbol{L}^{(p)}_{c_1,c_2} := c_1\left(\boldsymbol{M}^{(1,p)} + \boldsymbol{\Delta}^{(1,p)}\right) + c_2\boldsymbol{\Delta}^{(2,p)}$.
\end{proof}

\section{Proofs of the Main Results}\label{sec:proofs}

All of the proofs use the following basic facts. For $\boldsymbol{u} := \begin{pmatrix}
    1 & -1
\end{pmatrix}^{\top}$, 
\begin{align}\label{eq:main}
   \frac{S_n}{n} - \Lambda_p = {\boldsymbol{u}^\top}\left(\boldsymbol{\Gamma}_n - \boldsymbol{\Gamma}^{(p)}\right).
\end{align}
For $p \leq 1/2$, $\boldsymbol{u} = \boldsymbol{\nu}^{(2)}_p\left(\boldsymbol{\Gamma}^{(p)}\right)$ and  for $1/2 < p < 7/8$, $\boldsymbol{u} = \boldsymbol{\nu}^{(1)}_p\left(\boldsymbol{\Gamma}^{(p)}\right)$. Let 
\begin{align}\label{eq:al-be}
    \alpha = -\frac{4-5p+(2p-1)\sqrt{32p^2-52p+21
   }-\sqrt{-64p^3+161p^2-134p+37}}{2\sqrt{-64p^3+161p^2-134p+37}}, \quad \beta = 1 - \alpha,
\end{align}
so that for $p > 7/8$,
\begin{align}\label{eq:al-be2}
    \boldsymbol{u} = \alpha\boldsymbol{\nu}^{(1)}_p\left(\boldsymbol{\Gamma}^{(p)}\right) + \beta\boldsymbol{\nu}^{(1)}_p\left(\boldsymbol{\Gamma}^{(p)}\right).
\end{align}

\begin{proof}[Proof of Theorem~\ref{thm:slln-rw2mc}]
From \eqref{eq:main} and Theorem~\ref{thm:sa-slln}, \eqref{eq:s_p} is immediate, except the distribution of $\Lambda_p$. For that, observe
\[
   \Lambda_p \stackrel{d}{=} D_1\Lambda_p^{(+)} + D_2\Lambda_p^{(-)} + D_3\Lambda_p^{(\pm)} + D_4\Lambda_p^{(\mp)},
\]
where $\Lambda_p^{(+)}$ (respectively, $\Lambda_p^{(-)}$, $\Lambda_p^{(\pm)}$ and $\Lambda_p^{(\mp)}$) is the almost sure limit of $S_n/n$ conditioned on $X_1 = X_2 = 1$ (respectively, $X_1 = X_2 = -1$, $X_1 = -X_2 = 1, $ and $-X_1 = X_2 = 1$) and $\boldsymbol{D}$ is an independent $\operatorname{Multinomial \left(1;1/4,1/4,1/4,1/4\right)}$ random variable. As $\Lambda_p^{(+)} \stackrel{a.s.}{=} -\Lambda_p^{(-)}$ and $\Lambda_p^{(\pm)} \stackrel{a.s.}{=} -\Lambda_p^{(\mp)}$, it follows that $\mathbb{E}\left(\Lambda_p\right) = 0$. This completes the proof.
\end{proof}

\begin{remark}
For $p < 7/8$, the almost sure convergence of $S_n/n$ can also be established using the relation between the random walk with two memory channels with the multidimensional generalized elephant random walk (see Section~\ref{eg:mgerw}) and Theorem 3.6 of \cite{gerw}. However, the same will not work for $p > 7/8$ as the assumptions of Theorem 3.6 of \cite{gerw} no longer hold in this case.
\end{remark}

\begin{proof}[Proof of Theorem~\ref{thm:clt-rw2mc}]
If $p \in \left(0, \frac{11}{16}\right)\cup\left(\frac{113+\sqrt{97}}{128}, 1\right)$, then, from \eqref{eq:evalues}, we get $\lambda^{(1)}_p\left(\boldsymbol{\Gamma}^{(p)}\right) < -\frac{1}{2}$ almost surely. So, \eqref{eqthm:clt1} follows from \eqref{eq:mclt1}, \eqref{eq:main}, \eqref{eq:al-be2} and the fact that \[{\Sigma}^{(1)}_{0,1,p} = \frac{2}{11-16p}, \quad p \leq \frac{1}{2}, \quad {\Sigma}^{(1)}_{1,0,p} = \frac{2}{11-16p}, \quad p > \frac{1}{2}.\]
Similarly, if $p \in \left\{\frac{11}{16},\frac{113+\sqrt{97}}{128}\right\}$, then, from \eqref{eq:evalues}, we get $\lambda^{(2)}_p\left(\boldsymbol{\Gamma}^{(p)}\right) < \lambda^{(1)}_p\left(\boldsymbol{\Gamma}^{(p)}\right) = -\frac{1}{2}$ almost surely. So, \eqref{eqthm:clt2} follows from \eqref{eq:mclt2}, \eqref{eq:main}, \eqref{eq:al-be2} and the fact that ${\Sigma}^{(1)}_{1,0,11/16} = \frac{2}{3}$. Finally, if $p \in \left(\frac{11}{16}, \frac{7}{8}\right)\cup\left(\frac{7}{8}, \frac{113+\sqrt{97}}{128}\right)$, then, from \eqref{eq:evalues}, we get $\lambda^{(2)}_p\left(\boldsymbol{\Gamma}^{(p)}\right) < -\frac{1}{2} < \lambda^{(1)}_p\left(\boldsymbol{\Gamma}^{(p)}\right) < 0$ almost surely. So, \eqref{eqthm:clt3} follows from \eqref{eq:evalues}, \eqref{eq:mclt3}, \eqref{eq:main} and \eqref{eq:al-be2}.
\end{proof}

\begin{remark}
For $p < 7/8$, the central limit behavior of $S_n/n - \Lambda_p$ can also be established using the relation between the random walk with two memory channels with the multidimensional generalized elephant random walk (see Section~\ref{eg:mgerw}) and Theorem 3.13 of \cite{gerw}. However, the same will not work for $p > 7/8$ as the assumptions of Theorem 3.13 of \cite{gerw} no longer hold in this case.
\end{remark}

\begin{proof}[Proof of Theorem~\ref{thm:lil-rw2mc}]
If $p \in \left(0, \frac{11}{16}\right)$, then, from \eqref{eq:evalues}, we get $\lambda^{(1)}_p\left(\boldsymbol{\Gamma}^{(p)}\right) < -\frac{1}{2}$ almost surely. So, \eqref{eqthm:lil1-a} follows from \eqref{eq:mlil12}, \eqref{eq:mlil13}, \eqref{eq:main} and the fact that \[{\Sigma}^{(1)}_{0,1,p} = \frac{2}{11-16p}, \quad p \leq \frac{1}{2}, \quad {\Sigma}^{(1)}_{1,0,p} = \frac{2}{11-16p}, \quad p > \frac{1}{2}.\]
Similarly, $p \in \left(\frac{113+\sqrt{97}}{128}, 1\right)$, then also $\lambda^{(1)}_p\left(\boldsymbol{\Gamma}^{(p)}\right) < -\frac{1}{2}$ almost surely. So, \eqref{eqthm:lil1-b} follows from \eqref{eq:mlil11}, \eqref{eq:main} and \eqref{eq:al-be2}. Finally, if $p \in \left\{\frac{11}{16},\frac{113+\sqrt{97}}{128}\right\}$, then, from \eqref{eq:evalues}, we get $\lambda^{(2)}_p\left(\boldsymbol{\Gamma}^{(p)}\right) < \lambda^{(1)}_p\left(\boldsymbol{\Gamma}^{(p)}\right) = -\frac{1}{2}$ almost surely. So, \eqref{eqthm:clt2} follows from \eqref{eq:mclt2}, \eqref{eq:main} and \eqref{eq:al-be2} and the fact that ${\Sigma}^{(1)}_{1,0,11/16} = \frac{2}{3}$. 
\end{proof}

\begin{proof}[Proof of Theorem~\ref{thm:qsl-rw2mc}]
If $p \in \left(0, \frac{11}{16}\right)\cup\left(\frac{113+\sqrt{97}}{128}, 1\right)$, then, from \eqref{eq:evalues}, we get $\lambda^{(1)}_p\left(\boldsymbol{\Gamma}^{(p)}\right) < -\frac{1}{2}$ almost surely. So, \eqref{eqthm:qsl1} follows from \eqref{eq:mqsl1}, \eqref{eq:main}, \eqref{eq:al-be2} and the fact that \[{\Sigma}^{(1)}_{0,1,p} = \frac{2}{11-16p}, \quad p \leq \frac{1}{2}, \quad {\Sigma}^{(1)}_{1,0,p} = \frac{2}{11-16p}, \quad p > \frac{1}{2}.\]
Similarly, if $p \in \left\{\frac{11}{16},\frac{113+\sqrt{97}}{128}\right\}$, then, from \eqref{eq:evalues}, we get $\lambda^{(2)}_p\left(\boldsymbol{\Gamma}^{(p)}\right) < \lambda^{(1)}_p\left(\boldsymbol{\Gamma}^{(p)}\right) = -\frac{1}{2}$ almost surely. So, \eqref{eqthm:qsl2} follows from \eqref{eq:mqsl2}, \eqref{eq:main}, \eqref{eq:al-be2} and the fact that ${\Sigma}^{(1)}_{1,0,11/16} = \frac{2}{3}$.
\end{proof}

\section*{Acknowledgements} The authors would like to thank Surajit Saha for many useful discussions. The second and third authors are thankful to the organisers of \emph{StatPHYS Kolkata XII} for the kind hospitality at the \emph{S.~N.~Bose National Center for Basic Sciences}, where this project was initiated. 


\end{document}